\documentclass{amsart}

\usepackage{amsfonts}
\usepackage{amsthm}
\usepackage{amsmath, amssymb}
\usepackage{mathtools} 
\usepackage{thm-restate}
\usepackage{enumerate}
\usepackage{graphicx}
\usepackage{color}
\usepackage[percent]{overpic}
\usepackage[all]{xy}
\usepackage[colorlinks,linkcolor=blue,citecolor=blue]{hyperref}

\usepackage[top=3cm, bottom=3cm, left=3cm, right=3cm]{geometry}
\newtheorem{thm}{Theorem}
\numberwithin{thm}{section}
\newtheorem{mydef}[thm]{Definition}
\newtheorem{prop}[thm]{Proposition}

\newtheorem{lemma}[thm]{Lemma}
\newtheorem{ques}[thm]{Question}
\newtheorem{eg}[thm]{Example}

\newtheorem{conj}[thm]{Conjecture}

\newtheorem{remark}[thm]{Remark}
\newtheorem*{claim}{Claim}

\numberwithin{equation}{section}

\newcommand{\Z}{\mathbb{Z}}

\newcommand{\Q}{\mathbb{Q}}

\newcommand{\spin}{\ifmmode{\rm Spin}\else{${\rm spin}$\ }\fi}
\newcommand{\spinc}{\ifmmode{{\rm Spin}^c}\else{${\rm spin}^c$}\fi}
\newcommand{\spinct}{\mathfrak t}
\newcommand{\spincs}{\mathfrak s}
\newcommand{\Char}{{\rm Char}}
\newcommand{\norm}[1]{\lVert #1 \rVert^2}
\newcommand{\lone}[1]{\lVert #1 \rVert_1}
\newcommand{\wt}[1]{\widetilde{#1}}

\begin{document}

\title[Non-integer Seifert fibered surgeries]{The realization problem for non-integer Seifert fibered surgeries}

\author{Ahmad Issa}
\address{Department of Mathematics \\
         University of British Columbia \\
         Vancouver, BC \\
         Canada}
\email{aissa@math.ubc.ca}

\author{Duncan McCoy}
\address{D\'{e}partment de Math\'{e}matiques\\
  Universit\'{e} du Qu\'{e}bec \`{a} Montr\'{e}al\\
  Montr\'{e}al QC\\
  Canada}
\email{mc\_coy.duncan@uqam.ca}

\begin{abstract}
Conjecturally, the only knots in $S^3$ with non-integer surgeries producing Seifert fibered spaces are torus knots and cables of torus knots. In this paper, we make progress on the associated realization problem. Let $Y$ be a small Seifert fibered space arising by $p/q$-surgery on a knot in $S^3$, where $p/q$ is positive and a non-integer. Let $e$ denote the weight of the central vertex in the minimal star-shaped plumbing that $Y$ bounds. We show that if $e\leq -2$ or $e\geq 3$, then $Y$ can be obtained by $p/q$-surgery on a torus knot or a cable of a torus knot.
\end{abstract}
\maketitle
\section{Introduction}\label{sec:intro}
One of the simplest operations to produce new 3-manifolds is Dehn surgery on a knot $K$ in $S^3$. Thus, it is natural to consider how certain 3-manifolds may arise by surgery on a knot in $S^3$. It is, of course, well known that every closed oriented 3-manifold arises by surgery on a {\em link} in $S^3$ \cite{Lickorish62dehnsurgery, wallace60modifications}. One naturally arising family of 3-manifolds that might be considered for such questions are the Seifert fibered spaces.
\begin{ques}\label{ques}
Which Seifert fibered spaces can arise by surgery on a knot in $S^3$?
\end{ques}
As Seifert fibered spaces are not hyperbolic 3-manifolds, this is naturally related to the problem of understanding exceptional surgeries on hyperbolic knots in $S^3$. One conjecture is the following, which explains why one might consider integer and non-integer Seifert fibered surgeries separately.

\begin{conj}\cite[Conjecture~4.8]{gordon1998dehn}\label{conj:hyp_vers}
If $S^3_{p/q}(K)$ is a Seifert fibered space and $K$ is a hyperbolic knot, then $q=1$.
\end{conj}
This has an equivalent formulation which provides a conjectural list of knots in $S^3$ with non-integer Seifert fibered surgeries; see Proposition~\ref{prop:conjectures_equiv} for a proof of the equivalence.
\begin{conj}\label{conj:non_hyp_vers}
If $S^3_{p/q}(K)$ is a Seifert fibered space and $q\geq 2$, then $K$ is a torus knot or a cable of a torus knot.
\end{conj}
In this paper we consider Question~\ref{ques} for non-integer surgeries and show that for a significant subset of the Seifert fibered spaces the only ones arising by non-integer surgery on a knot in $S^3$ are the ones predicted by Conjecture~\ref{conj:non_hyp_vers}.

Culler-Gordon-Luecke-Shalen's cyclic surgery theorem shows that lens spaces arise by non-integer surgery only on torus knots \cite{cglscyclic}. Boyer and Zhang have shown that Haken Seifert fibered spaces can arise only by integer surgeries on knots in $S^3$ \cite[Corollary~J]{Boyer94exceptional}, a fact that also follows from later work of Gordon and Luecke \cite{Gordon04toroidal}. Thus it remains to consider non-integer surgeries yielding small Seifert fibered spaces, that is spaces that fiber over $S^2$ with three exceptional fibers.  We use $Y\cong S^2(e;\frac{p_1}{q_1}, \frac{p_2}{q_2}, \frac{p_3}{q_3})$ to denote the Seifert fibered space obtained according to the surgery diagram in Figure~\ref{fig:sfs_surgery_diagram}. If $Y$ is a rational homology sphere, then it arises as the boundary of a definite manifold obtained by plumbing sphere bundles according to a star-shaped graph. We define $e(Y)\in \Z \setminus \{0\}$ to be the weight of the central vertex of the unique minimal definite plumbing which $Y$ bounds; see Section~\ref{subsec:mindef}.
\begin{restatable}{thm}{thmnonint}
\label{thm:nonint}
Let $Y$ be a Seifert fibered space over $S^2$ with three exceptional fibers and $e(Y)\not\in \{+1, +2, -1\}$. If there is a knot $K$ in $S^3$ with $Y\cong S_{p/q}^3(K)$ where $p/q>0$ and $p/q \in \Q \setminus \Z$, then there is a knot $K'$ which is either a torus knot or a cable of a torus knot with $S_{p/q}^3(K')\cong Y$ and $\Delta_K(t)=\Delta_{K'}(t)$.
\end{restatable}

It turns out that the spaces arising in the conclusion of Theorem~\ref{thm:nonint} are all $L$-spaces. Thus, the fact that $K$ and $K'$ have the same Alexander polynomial shows that they have isomorphic knot Floer homology groups \cite{Ozsvath05Lensspace}. In order to make full use of Theorem~\ref{thm:nonint}, one also needs to understand for which surgeries on torus knots or cables of torus knots we have $e(Y)\not\in \{+1, +2, -1\}$. Thus, we provide the following result as a companion to Theorem~\ref{thm:nonint}.
\begin{restatable}{prop}{ecalc}
\label{prop:e_calc}
Let $K$ be a torus knot or a cable of a torus knot. Then for $p/q>0$ we have that $S_{p/q}^3(K)$ is a Seifert fibered space over $S^2$ with three exceptional fibers and $e(S_{p/q}^3(K))\not\in \{-1,+1,+2\}$ if and only if
\begin{enumerate}[(i)]
\item $K$ is a torus knot $K=T_{r,s}$ with $r,s>1$, $p/q>rs-1$ and $|p-rsq|>1$ or
\item $K$ is a cable of a torus knot $K=C_{a,b} \circ T_{r,s}$, where $r,s>1$, $b/a>rs-1$ and $p/q=ab \pm 1/q$
\end{enumerate}
\end{restatable}

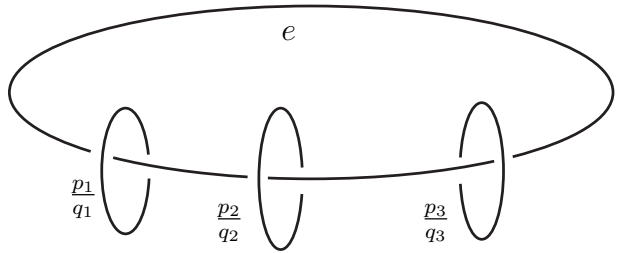
\begin{figure}
  \centerline{
    \begin{overpic}[width=0.5\textwidth]{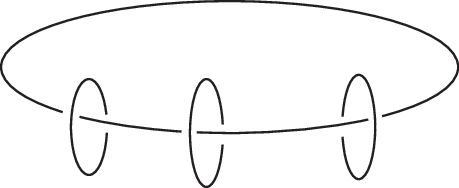}
      \put (45,35) {\large $e$}
      \put (10,8) {\large $\frac{p_1}{q_1}$}
      \put (34,4) {\large $\frac{p_2}{q_2}$}
      \put (68,4) {\large $\frac{p_3}{q_3}$}
    \end{overpic}
  }
\caption{Surgery diagram of the Seifert fibered space $S^2(e; \frac{p_1}{q_1}, \frac{p_2}{q_2}, \frac{p_3}{q_3})$.}
\label{fig:sfs_surgery_diagram}
\end{figure}

Although Theorem~\ref{thm:nonint} is phrased in terms of positive surgeries, we will instead reflect $Y$, if necessary, to assume that it bounds a positive definite plumbing, i.e so that $e(Y)\geq 2$. Thus in order to prove Theorem~\ref{thm:nonint} we have two possible cases to consider. Either we have $e(Y)=2$ and $p/q<0$ or we have that $e(Y)\geq 3$. We deal with these two regimes differently. The main technical content of this paper comes in the analysis of the $e(Y)=2$ case. The key point is that the definite plumbing bounding a Seifert fibered space is an example of a ``sharp'' manifold, meaning that, roughly speaking, its intersection form determines the Heegaard Floer $d$-invariants of its boundary \cite{Ozsvath03plumbed}. This allows us to apply the changemaker lattice surgery obstruction developed by Greene for integer and half-integer surgeries \cite{greene2010space, Greene3Braid}, and extended to all non-integer surgeries by Gibbons \cite{gibbons2013deficiency}. This reduces the problem to studying when the intersection form of a star-shaped plumbing can be isomorphic to a changemaker lattice. Almost all previous applications of changemaker lattices have involved studying situations in which changemaker lattices are isomorphic to graph lattices. However, when $e(Y)=2$ the intersection form of the relevant star-shaped plumbing is not a graph lattice, meaning that new ideas are required to apply the changemaker obstruction. The majority of the technical innovation in this paper comes from circumventing the fact that we are not dealing with a graph lattice. When $e(Y)\geq 3$ the Seifert fibered space is the branched double cover of an alternating Montesinos link. This allows us to apply previous results describing when the double branched cover of an alternating link can arise by non-integer surgery \cite{mccoy2014noninteger}. Although the results of \cite{mccoy2014noninteger} were derived using changemaker lattices, we do not explicitly use lattice theoretic techniques in this part of the proof. We prove the theorem by considering Conway spheres in alternating diagrams of Montesinos links. 

The structure of the paper is as follows. We begin in Section~\ref{sec:SF_surgeries} by recalling some properties of Seifert fibered surgeries and observing that Conjecture~\ref{conj:non_hyp_vers} is true for surgeries with $q\geq 9$. Sections~\ref{sec:SF_plumbings}~and~\ref{sec:CM_lattices} contain the necessary background on lattices, with Section~\ref{sec:SF_plumbings} discussing the necessary results on the intersection forms of plumbings and Section~\ref{sec:CM_lattices} addressing changemaker lattices. The technical results necessary for the $e(Y)=2$ case of Theorem~\ref{thm:nonint} are developed in Section~\ref{sec:non_int_surgeries}. The $e(Y)\geq 3$ case is studied in Section~\ref{sec:alt_case}. Finally, in Section~\ref{sec:proof_of_main}, we pull together all the necessary results to prove Theorem~\ref{thm:nonint} and Proposition~\ref{prop:e_calc}.

\subsection*{Acknowledgements} The first author would like to thank Cameron Gordon for his support, encouragement and helpful conversations. The second author has been thinking (mostly unsuccessfully) about questions relating to this paper for a number of years. He is grateful to numerous people, most notably Josh Greene and Brendan Owens, for helpful conversations during various incarnations of this project. Both authors would like to thank the anonymous referee for their detailed feedback. 

\section{Seifert fibered surgeries}\label{sec:SF_surgeries}
In this section we justify the equivalence of Conjecture~\ref{conj:hyp_vers} and Conjecture~\ref{conj:non_hyp_vers}. We also note that Conjecture~\ref{conj:hyp_vers} is true for $q\geq 9$.
\begin{lemma}\label{lem:increasing_q}
Let $K$ be a knot which is not a torus knot or a cable of a torus knot with a Seifert fibered surgery $S_{p/q}^3(K)$ for some $q\geq 2$. Then there is a hyperbolic knot $K'$ and $q'\geq q$ such that $S_{p/q}^3(K)\cong S_{p/q'}^3(K')$.
\end{lemma}
\begin{proof}
  By Thurston's work every knot is either a hyperbolic knot, a satellite knot or a torus knot \cite{Riley79hyperbolic, Thurston823DKleinanGroups}. Applied to $K$, this shows that $K$ is a hyperbolic knot or a satellite knot. If $K$ is hyperbolic then we may take $K'=K$ and $q'=q$. Thus suppose that $K$ is a satellite knot. Consider an innermost incompressible torus $R$ in $S^3 \setminus \nu K$. This cuts $S^3 \setminus \nu K$ into two components, one of these is the complement of a knot $K'\subset S^3$ and on the other side the complement of a knot $C \subseteq S^1 \times D^2$ in a solid torus. The innermost assumption on $R$ implies that $K'$ is either a hyperbolic knot or a torus knot. Since $S_{p/q}^3(K)$ is a small Seifert fibered space \cite[Corollary~J]{Boyer94exceptional}, it is irreducible and atoroidal. Therefore after performing surgery the torus $R$ must bound a solid torus. In particular, $C$ must be a knot in $S^1 \times D^2$ with a non-trivial $S^1\times D^2$ surgery. By the work of Gabai \cite[Lemma~2.3]{gabai89solidtori}, this implies that $C$ is either a torus knot or a 1-bridge braid in the solid torus.
However Gabai has also shown that 1-bridge braids admit only integer solid torus surgeries \cite[Lemma~3.2]{gabai90onebridge}. Thus $C$ must a torus knot in $S^1 \times D^2$. This implies that $K$ is a cable of $K'$. As we are assuming that $K$ is not a cable of a torus knot, it follows that $K'$ is a hyperbolic knot.
Since the torus $R$ bounds a solid torus after performing surgery on $C$ it follows that $S^3_{p/q}(K) \cong S^3_{p'/q'}(K')$, where $p'/q'$ is the slope on $R$ which bounds a disk after this surgery. By considering how the homology of a solid torus changes under surgery  one can see that $p'/q'=p/(qw^2)$, where $w\geq 2$ is the winding number of $C$ \cite[Lemma~3.3]{Gordon83Satellite}. This completes the proof of the lemma.
\end{proof}
This allows us to prove the following two useful results.
\begin{prop}\label{prop:conjectures_equiv}
Conjecture~\ref{conj:hyp_vers} $\Leftrightarrow$ Conjecture~\ref{conj:non_hyp_vers}
\end{prop}
\begin{proof}
The implication Conjecture~\ref{conj:hyp_vers} $\Leftarrow$ Conjecture~\ref{conj:non_hyp_vers} follows from the fact that torus knots and cables of torus knots are not hyperbolic knots.
 The implication Conjecture~\ref{conj:hyp_vers} $\Rightarrow$ Conjecture~\ref{conj:non_hyp_vers} follows from Lemma~\ref{lem:increasing_q} since Conjecture~\ref{conj:hyp_vers} asserts that no hyperbolic knot $K'$ satisfying the conclusion of the lemma can exist.
\end{proof}
\begin{prop}\label{prop:q>9_case}
If $S_{p/q}^3(K)$ is a Seifert fibered space and $q\geq 9$ then $K$ is a cable of a torus knot or a torus knot.
\end{prop}
\begin{proof}
Lackenby and Meyerhoff have shown that the distance between exceptional fillings on a hyperbolic knot is eight \cite{Lackenby-meyerhoff-13}. Therefore if $K'$ is a hyperbolic knot such that $S_{p/q'}^3(K')$ is a Seifert fibered space, then $q' \leq 8$. Hence the proposition follows from Lemma~\ref{lem:increasing_q}.
\end{proof}

\section{Seifert fibered spaces and plumbings}\label{sec:SF_plumbings}
In this paper we use $S^2(e; \frac{p_1}{q_1}, \frac{p_2}{q_2}, \frac{p_3}{q_3})$ to denote the space obtained by surgery on the link as in Figure~\ref{fig:sfs_surgery_diagram}, where $e\in \Z$ and for each $i$, $p_i$ and $q_i$ are coprime integers. This is a Seifert fibered space with three exceptional fibers provided that $|p_i|>1$ for $i=1,2,3$. By performing Rolfsen twists on the $p_i/q_i$-framed components, we see that there is an orientation preserving homeomorphism between
$S^2\left(e; \frac{p_1}{q_1}, \frac{p_2}{q_2},\frac{p_3}{q_3}\right)$ and $S^2\left(e'; \frac{p_1'}{q_1'}, \frac{p_2'}{q_2'},\frac{p_3'}{q_3'}\right)$
whenever
\begin{equation}\label{eq:homeo1}
e-\frac{q_1}{p_1}- \frac{q_2}{p_2}-\frac{q_3}{p_3}= e'- \frac{q_1'}{p_1'}- \frac{q_2'}{p_2'}-\frac{q_3'}{p_3'}
\end{equation}
and there is a permutation $\pi$ of $\{1,2,3\}$ such that 
\begin{equation}\label{eq:homeo2}
\frac{q_i}{p_i} \equiv \frac{q_{\pi(i)}'}{p_{\pi(i)}'} \bmod 1 \quad\text{for $i=1,2,3$.}
\end{equation}
Conversely it follows from the classification of Seifert fibered space (see, for example, the results in \cite[Section~5.3]{Orlik}) 
that conditions \eqref{eq:homeo1} and \eqref{eq:homeo2} are, in fact, necessary for there to be an orientation preserving homeomorphism between $S^2\left(e; \frac{p_1}{q_1}, \frac{p_2}{q_2},\frac{p_3}{q_3}\right)$ and $S^2\left(e'; \frac{p_1'}{q_1'}, \frac{p_2'}{q_2'},\frac{p_3'}{q_3'}\right)$.
The generalized Euler invariant of $Y\cong S^2(e; \frac{p_1}{q_1}, \frac{p_2}{q_2}, \frac{p_3}{q_3})$ is defined to be
\[
\varepsilon(Y):= e - \frac{q_1}{p_1}-\frac{q_2}{p_2}-\frac{q_3}{p_3}.
\]
By the above discussion, one sees that $\varepsilon(Y)$ is a topological invariant. Reversing the orientation on the Seifert fibered space $Y$ yields the Seifert fibered space
\[
-Y\cong S^2\left(-e; -\frac{p_1}{q_1}, -\frac{p_2}{q_2},-\frac{p_3}{q_3}\right).
\]
Thus we see that the generalized Euler characteristic satisfies
\[
\varepsilon(-Y)= -\varepsilon(Y)
\]
Using the surgery description of $Y$ in Figure~\ref{fig:sfs_surgery_diagram} one finds that the order of its first homology can be calculated as:
\[
|H_1 (Y;\Z)| = |(p_1 p_2 p_2) \varepsilon(Y)|.
\]
It follows that if $Y$ is a rational homology sphere if and only if $\varepsilon(Y) \neq 0$. Thus if $Y$ is a Seifert fibered space rational homology sphere, then $Y$ can be oriented so that $\varepsilon(Y)>0$. 

\subsection{Minimal definite plumbings}\label{subsec:mindef}
Let $Y$ be a Seifert fibered rational homology sphere with three exceptional fibers oriented so that $\varepsilon(Y)>0$. The discussion at the start of this section shows that $Y$ has a unique description in the form
\[
Y\cong S^2\left(e; \frac{p_1}{q_1}, \frac{p_2}{q_2}, \frac{p_3}{q_3}\right),
\]
where $e>0$ and $\frac{p_1}{q_1}, \frac{p_2}{q_2}, \frac{p_3}{q_3}>1$. We define $e(Y)$ to let the value $e$ in this presentation with the convention that $e(-Y)=-e(Y)$. This is the invariant $e(Y)$ appearing in the statement of Theorem~\ref{thm:nonint}. As we will see, this quantity is precisely the weight of the central vertex in the minimal definite plumbing that $Y$ bounds.

There is a unique continued fraction expansion
\[
\frac{p_1}{q_1} = [a_1, \ldots, a_k]^{-} = a_1 - \cfrac{1}{a_2 - \cfrac{1}{\begin{aligned}\ddots \,\,\, & \\[-3ex] & a_{k-1} - \cfrac{1}{a_{k}} \end{aligned}}}\;,
\]
where $k \ge 1$ and $a_j \ge 2$ for all $j \in \{1, \ldots, k\}$. Similarly, we may write $\frac{p_2}{q_2} = [b_1, \ldots, b_l]^{-}$ and $\frac{p_3}{q_3} = [c_1, \ldots, c_m]^{-}$, where $l, m \ge 1$ and $b_j\ge 2$ and $c_j\ge 2$ for all $j$.
By performing a sequence of reverse slam-dunks to convert the fractional surgery coefficients to integer coefficients, we see that $Y$ has a surgery description as shown in Figure~\ref{fig:corresponding_surgery}. Since these surgery coefficients are integers, this can also be viewed as a Kirby diagram for a 4-manifold $X$ with $\partial X = Y$. This manifold is diffeomorphic to one obtained by plumbing disk bundles over $S^2$ according to the star-shaped graph given in Figure~\ref{fig:star_plumbing} (see Section~6.1 of \cite{GompfStipsicz}). This $X$ is precisely the unique minimal positive definite plumbing that $Y$ bounds \cite[Theorem 5.2]{MR518415}. Given the plumbing diagram as in Figure~\ref{fig:star_plumbing} we can define an integer lattice $(\Lambda_\Gamma, Q_\Gamma)$, where $\Lambda_\Gamma$ is the free abelian group generated by the vertices of $\Gamma$ and $Q_\Gamma \colon \Lambda_\Gamma \times \Lambda_\Gamma \rightarrow \Z$ is the bilinear pairing with
\[
Q_\Gamma(u,v) = \begin{cases}
      \text{w}(v), & \mbox{if }u = v \\
      -1, & \mbox{if vertices }u\mbox{ and }v\mbox{ are connected by an edge} \\
      0, & \mbox{otherwise}
   \end{cases},
\]
where $u$ and $v$ are vertices of $\Gamma$ and $\text{w}(v)$ denotes the weight of vertex $v$. The lattice $(\Lambda_\Gamma, Q_\Gamma)$ is naturally isomorphic to the intersection form of $X$, and hence is positive definite. We will write $x \cdot y$ to denote the pairing $Q_X(x, y)$ and $\norm{x}$ to denote $Q_X(x,x)$.

\begin{figure}[!ht]
  \centerline{
    \begin{overpic}[width=0.5\textwidth]{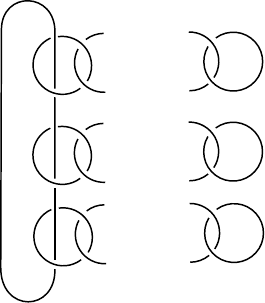}
      \put (-4,50) {$e$}
      \put (21,91) {$a_1$}
      \put (34,91) {$a_2$}
      \put (55,91) {$a_{k-1}$}
      \put (86,87) {$a_k$}
      \put (21,60) {$b_1$}
      \put (34,60) {$b_2$}
      \put (55,60) {$b_{l-1}$}
      \put (86,57) {$b_l$}
      \put (21,34) {$c_1$}
      \put (34,34) {$c_2$}
      \put (55,34) {$c_{m-1}$}
      \put (86,30) {$c_m$}
      \put (45,78) {$\dots$}
      \put (45,50) {$\dots$}
      \put (45,25) {$\dots$}
    \end{overpic}
  }
\caption{The Kirby diagram for $X$ (also a surgery diagram for $\partial X = Y$).}
\label{fig:corresponding_surgery}
\end{figure}

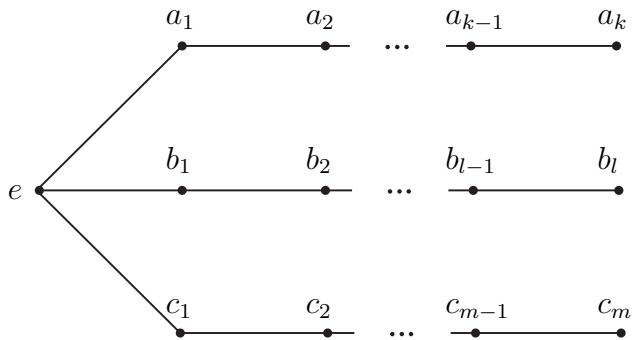
\begin{figure}[!ht]
  \centerline{
    \begin{overpic}[width=0.5\textwidth]{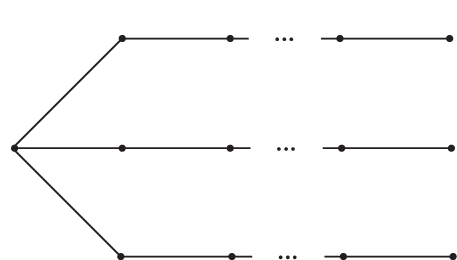}
      \put (-2,23) {\large $e$}
      \put (24,52) {\large $a_1$}
      \put (47,52) {\large $a_2$}
      \put (70,52) {\large $a_{k-1}$}
      \put (95,52) {\large $a_k$}
      \put (24,28) {\large $b_1$}
      \put (47,28) {\large $b_2$}
      \put (70,28) {\large $b_{l-1}$}
      \put (95,28) {\large $b_l$}
      \put (24,4) {\large $c_1$}
      \put (47,4) {\large $c_2$}
      \put (70,4) {\large $c_{m-1}$}
      \put (95,4) {\large $c_m$}
    \end{overpic}
  }
\caption{Weighted star-shaped plumbing graph $\Gamma$.}
\label{fig:star_plumbing}
\end{figure}

\subsection{Quasi-alternating plumbings}
In order to prove Theorem~\ref{thm:nonint} we need to understand the properties of lattices arising as the intersection forms in the case $e=2$. For topological reasons we need only consider a special subset of such forms. The following was proven by the first author in his classification of quasi-alternating Montesinos links \cite{issa2017quasialternating}.

\begin{lemma}\label{lem:QAchar}
 Let $Y = S^2(e; \frac{p_1}{q_1}, \frac{p_2}{q_2}, \frac{p_3}{q_3})$, $e\ge 2$, be such that $Y$ is the boundary of the (canonical) positive definite plumbing $4$-manifold $X$. Then the following are equivalent:
\begin{enumerate}[(i)]
\item\label{it:neg_def} $Y$ bounds a negative definite $4$-manifold $W$ with $H_1(W)$ torsion free.
\item\label{it:dbc_QA} $Y$ is homeomorphic to the double branched cover of a quasi-alternating Montesinos link.
\item\label{it:num_condition} Either $e \ge 3$, or $e = 2$ and $\frac{q_i}{p_i} + \frac{q_j}{p_j} < 1$ for some $i, j \in \{1,2,3\}$ with $i \neq j$.
\item\label{it:embed_condition} If $A$ is a matrix representing some embedding $H_2(X) \xhookrightarrow{} \Z^n$, $n\in\Z_{>0}$, of the intersection lattice of $X$ into a standard positive diagonal lattice with respect to a pair of bases, then $A^T$ is surjective.
\end{enumerate}
\end{lemma}

On account of the condition Lemma~\ref{lem:QAchar}\eqref{it:dbc_QA} we make the following definition:
\begin{mydef}
Let $\Gamma$ be star-shaped plumbing graph as in Figure~\ref{fig:star_plumbing}. We say that $\Gamma$ is {\em quasi-alternating} if $e=2$ and the continued fractions
\[\frac{p_1}{q_1}=[a_1, \dots, a_k]^-\] 
and
\[\frac{p_2}{q_2}=[b_1, \dots, b_l]^-\]
satisfy $\frac{q_1}{p_1}+\frac{q_2}{p_2}<1$. We will also call the corresponding lattice $\Lambda_\Gamma$ {\em quasi-alternating}. 
\end{mydef}

In order to study quasi-alternating lattices, it will be convenient to define the following quadratic form.

\begin{mydef} Suppose $k > 0$ and $n_1, \ldots, n_k \in \Z$.
  We denote by $Q_{n_1, \dots, n_k}$ the quadratic form given by
\[
Q_{n_1, \dots, n_k}(x_1, \dots, x_k)=n_1 x_1^2-2x_1 x_2 + n_2 x_2^2 - \dots - 2x_{k-1}x_k + n_k x_k^2,
\]
  for all $x_1, \ldots, x_k \in \Z$.
\end{mydef}
We will begin by proving some preparatory inequalities on quadratic forms of this type.

\begin{lemma}\label{lem:qa_bound_1}
Let $c_1, \dots, c_m\geq 2$ be integers and $z_1, \dots, z_m$ be integers. We have the following inequalities:
\begin{enumerate}[(i)]
\item\label{it:alt_lin_chain}
If at least one $z_i$ is non-zero, then 
\[
Q_{c_1,\dots, c_m}(z_1,\dots, z_m)\geq 2+\sum_{i=1}^m(c_i-2)|z_i|
\]
\item\label{it:lin_chain_1}
\[
Q_{1,c_1,\dots, c_m}(c,z_1,\dots, z_m)\geq |c|+ \sum_{i=1}^m (c_i-2) |z_i|
\]
\item\label{it:lin_chain_c_nonzero}
If $c\neq 0$ or $z_i\neq 0$ for some $i$, then
\[
Q_{1,c_1,\dots, c_m}(c,z_1,\dots, z_m)+|c| \geq 2+ \sum_{i=1}^m (c_i-2) |z_i|.
\] 
\end{enumerate} 
\end{lemma}
\begin{proof}
We prove \eqref{it:alt_lin_chain} first. Since $c_i \ge 2$ for all $1 \le i \le m$ we can complete the square to obtain
\begin{equation}\label{eq:sq_com_1}
Q_{c_1,\ldots,c_m}(z_1,\ldots,z_m)= z_1^2 + (z_1 - z_2)^2 + \cdots + (z_{m-1} - z_m)^2 + z_m^2 + \sum_{i=1}^m z_i^2(c_i-2).
\end{equation}
If $z_i$ is non-zero for some $i$, then at least two of the terms
\[z_1^2,  (z_1 - z_2)^2, \dots, (z_{m-1} - z_m)^2, \quad \text{or}\quad z_m^2\]
must be non-zero. Since these terms are all integers, this gives the desired inequality when combined with \eqref{eq:sq_com_1}.

Now we prove \eqref{it:lin_chain_1} and \eqref{it:lin_chain_c_nonzero}.
 Since $c_i \ge 2$ for all $1 \le i \le m$, we can complete the square to obtain
\begin{equation}\label{eq:sq_com_2}
Q_{1,c_1,\ldots,c_m}(c,z_1,\ldots,z_m)=(c - z_1)^2 + (z_1 - z_2)^2 + \cdots + (z_{m-1} - z_m)^2 + z_m^2 + \sum_{i=1}^m z_i^2(c_i-2). 
\end{equation}
However notice that we have
\begin{align*}
(c - z_1)^2 + (z_1 - z_2)^2 + \cdots + (z_{m-1} - z_m)^2 + z_m^2 &\geq  |c - z_1| + |z_1 - z_2| + \cdots + |z_{m-1} - z_m| + |z_m|\\
&\geq |(c - z_1) + \cdots + (z_{m-1} - z_m) + z_m| = |c|.
\end{align*}
Combining this with \eqref{eq:sq_com_2} proves \eqref{it:lin_chain_1}.

To prove \eqref{it:lin_chain_c_nonzero} observe that if at least one of $c, z_1, \dots, z_m$ is non-zero then at least two of the terms
\[|c|, (c - z_1)^2,  (z_1 - z_2)^2, \dots, (z_{m-1} - z_m)^2, \quad\text{or}\quad z_m^2\]
must be non-zero. Since each of these terms are integers, this gives the desired inequality when combined with \eqref{eq:sq_com_2}. 
\end{proof}

\begin{lemma}\label{lem:qa_lb_technical} Let $a_1, \ldots, a_k, b_1,\ldots, b_l \ge 2$ be integers and let $\frac{p_1}{q_1}=[a_1,\dots, a_k]^-$ and $\frac{p_2}{q_2}=[b_1, \dots, b_l]^-$ where $(p_i, q_i) = 1$ and $p_i > q_i \ge 1$ for $i \in \{1,2\}$. Suppose that $\frac{q_1}{p_1} + \frac{q_2}{p_2} < 1$. Then for any integers $x_1,\ldots,x_k,y_1,\ldots,y_l,c$ with at least one of $c$ or the $x_i$ or $y_i$ non-zero, we have
\begin{equation}\label{eq:key_technical_ineq}
Q_{a_k,\dots, a_1,1,b_1, \dots,b_l}(x_k,\ldots,x_1,c,y_1,\ldots,y_l) + |c| \geq 2 + \sum_{i=1}^k (a_i-2)|x_i| + \sum_{i=1}^l (b_i - 2)|y_i|.
\end{equation}
\end{lemma}
\begin{proof}
First observe that
\[Q_{a_1,\dots, a_k,1,b_1,\dots, b_l}(0,\dots,0,c,0,\dots,0)+|c|=c^2+|c|.\]
So if $c\neq 0$ and $x_1=\dots = x_k=y_1=\dots= y_l=0$, then 
\[Q_{a_k,\dots, a_1,1,b_1, \dots,b_l}(x_k,\ldots,x_1,c,y_1,\ldots,y_l)\geq 2,
\]
which is the desired inequality.
Thus we may assume that at least one of the $x_i$ or $y_j$ terms is non-zero.

Consider the integers $x_1,\ldots,x_k,y_1,\ldots,y_l,c$. The right hand side of \eqref{eq:key_technical_ineq} is invariant under changing the signs of any subset of these integers. Moreover, the left hand side of \eqref{eq:key_technical_ineq} is minimal when all these integers have the same sign, and is invariant under simultaneously replacing all of the integers by their negatives. Hence, it suffices to consider the case $x_1,\ldots,x_k,y_1,\ldots,y_l,c\geq 0$.

Now consider the following inequality
\begin{align}\begin{split}\label{eq:monster_sq_comp}
  a_1x_1^2&-2x_1c + c^2 -2y_1c + b_1y_1^2 + c\\
     &= (a_1-1)x_1^2 -2x_1y_1 + (b_1-1)y_1^2+x_1+y_1 + (x_1+y_1 -c-\frac{1}{2})^2 - \frac{1}{4}\\
&\geq (a_1-1)x_1^2 -2x_1y_1 + (b_1-1)y_1^2+x_1+y_1,
\end{split}\end{align}
where the inequality follows from the observation that the square of a half integer is always at least a quarter. It follows from \eqref{eq:monster_sq_comp} that
\begin{equation}\label{eq:QA_induct}
  Q_{a_k,\dots, a_1,1,b_1, \dots,b_l}(x_l,\ldots,x_1,c,y_1,\ldots,y_l)+|c| \geq Q_{a_k,\dots, a_1-1,b_1-1, \dots,b_l}(\ldots,x_1,y_1,\ldots)+|x_1|+|y_1|,
\end{equation}
where we are using the positivity assumption to write $|x_1|, |y_1|$ and $|c|$ in place of $x_1, y_1$ and $c$. 

We will use \eqref{eq:QA_induct} to prove \eqref{eq:key_technical_ineq} by induction.

Note that the condition $\frac{q_1}{p_1} + \frac{q_2}{p_2} < 1$ implies that at most one of $a_1$ and $b_1$ can equal two.

If $a_1>2$ and $b_1>2$, then Lemma~\ref{lem:qa_bound_1}\eqref{it:alt_lin_chain} applies to show that
\[
Q_{a_k,\dots, a_1-1,b_1-1, \dots,b_l}(x_l,\ldots,x_1,y_1,\ldots,y_l)\geq 2+\sum |x_i|(a_i-2) -x_1+\sum |y_i|(b_i-2)-y_1.
\]
Combining this with \eqref{eq:QA_induct} gives the desired inequality.

Thus it suffices to consider the possibility that $a_1=2$ or $b_1=2$. Without loss of generality we can assume that $a_1=2$. If $k=1$, then Lemma~\ref{lem:qa_bound_1}\eqref{it:lin_chain_c_nonzero} combined with \eqref{eq:QA_induct} gives the desired bound.

Thus, it remains to consider the case that $a_1=2$ and $k>1$. Let 
\[
\frac{p_1'}{q_1'}=[a_2, \dots , a_k]^- 
\quad\text{and}\quad
\frac{p_2'}{q_2'} = [b_1-1,b_2,\dots, b_l]^-.
\]
We wish to show that these satisfy $\frac{q_1'}{p_1'}+\frac{q_2'}{p_2'}<1$. Since $a_1=2$, we have that $\frac{p_1}{q_1}=2-\frac{q_1'}{p_1'}$. We also have that $\frac{p_2'}{q_2'}=\frac{p_2}{q_2} -1$.  The condition that $\frac{q_1}{p_1}+\frac{q_2}{p_2}<1$ implies that $\frac{p_1}{q_1}>\frac{p_2}{p_2-q_2}$. Thus see that
\[
\frac{q_1'}{p_1'} + \frac{q_2'}{p_2'} =
2-\frac{p_1}{q_1}+\frac{q_2}{p_2-q_2}<
2-\frac{p_2}{p_2-q_2}+\frac{q_2}{p_2-q_2} =1,
\]
as required.

This allows us to prove the lemma inductively, by considering
\[
Q_{a_k,\dots, a_2,1,b_1-1,\dots,b_l}(x_k,\dots, x_1,y_1,\dots,y_l)
\]
with $x_1$ taking the role of $c$.
\end{proof}

With these inequalities in place, we can now prove the following, which will be our key result on quasi-alternating lattices.

\begin{lemma}\label{lem:key_inequality}
Let $\Lambda$ be a quasi-alternating lattice associated to a graph $\Gamma$ and let $V\subseteq \Lambda$ the basis elements corresponding to the vertices of $\Gamma$. Then for any non-zero $x=\sum_{v\in V} c_v v$, we have
\[\norm{x} \geq 2 + \sum_{v \in V} |c_v|(\norm{v} -2).\]
\end{lemma}
\begin{proof}
Suppose that $\Lambda$ is the lattice corresponding to the star-shaped plumbing in Figure~\ref{fig:star_plumbing} with $e=2$ and 
\[
\frac{p_1}{q_1} = [a_1, \dots, a_k]^{-},\quad
\frac{p_2}{q_2} = [b_1, \dots, b_l]^{-} 
\quad\text{and}\quad
\frac{p_3}{q_3} = [c_1, \dots, c_m]^{-}, 
\]
where $a_i,b_i,c_i\geq 2$ and $\frac{q_1}{p_1}+\frac{q_2}{p_2}<1$. Thus if we take
\[
c_v=
\begin{cases}
x_i &\text{$v$ is the $a_i$-weighted vertex}\\
y_i &\text{$v$ is the $b_i$-weighted vertex}\\
z_i &\text{$v$ is the $c_i$-weighted vertex}\\
c   &\text{$v$ is the central vertex,}\\
\end{cases}
\]
then it is not hard to verify that $\norm{x}$ can be calculated as
\begin{equation}\label{eq:Q_splitting}
\norm{x}=
Q_{a_k, \dots, a_1,1,b_1,\dots, b_l}(x_k, \dots, x_1,c,y_1,\dots, y_l)+Q_{1,c_1,\dots,c_m}(c,z_1,\dots, z_m).
\end{equation}

If $c=0$, then \eqref{eq:Q_splitting} simplifies to
\[
\norm{x}=Q_{a_k, \dots, a_1}(x_k, \dots, x_1)+Q_{b_1,\dots, b_l}(y_1,\dots, y_l)+Q_{c_1,\dots,c_m}(z_1,\dots, z_m).
\]
In this case the required inequality follows from Lemma~\ref{lem:qa_bound_1}\eqref{it:alt_lin_chain}.

Thus it suffices to suppose that $c\neq 0$. In this case, we can apply Lemma~\ref{lem:qa_bound_1}\eqref{it:lin_chain_1} to the second summand of \eqref{eq:Q_splitting}. This gives
\[
\norm{x}\geq Q_{a_k, \dots, a_1,1,b_1,\dots, b_l}(x_k, \dots, x_1,c,y_1,\dots, y_l)+ |c| + \sum_{i=1}^m |z_i|(c_i-2).
\]
By applying Lemma~\ref{lem:qa_lb_technical}, we get the desired inequality.
\end{proof}

Lemma~\ref{lem:key_inequality} has several consequences that will be of use later. To describe these consequences we need the following lattice-theoretic concepts.

\begin{mydef}\label{def:irred/unbreak}
Let $\Lambda$ be an integer lattice and let $v\in \Lambda$.
\begin{itemize}
\item The vector $v$ is {\em irreducible} if for all $x,y\in\Lambda$, $v = x + y$ and $x\cdot y \geq 0$ implies either $x=0$ or $y=0$.
\item The vector $v$ is {\em unbreakable} if for all $x,y\in\Lambda$, $v = x + y$ and $x\cdot y =-1$ implies either $\norm{x}=2$ or $\norm{y}=2$.
\end{itemize}
\end{mydef}

\begin{lemma}\label{lem:key_properties}
Let $\Lambda$ be a quasi-alternating lattice associated to a graph $\Gamma$ and let $V\subseteq \Lambda$ the basis elements corresponding to the vertices of $\Gamma$. Then following are true:
\begin{enumerate}[(i)]
\item\label{it:no_norm_1} If $x\in \Lambda$ is non-zero, then $\norm{x}\geq 2$.
\item\label{it:norm_bound} If $x=\sum_{v\in V} c_v v$, then if $c_w\neq 0$ for some $w\in V$, then $\norm{x}\geq \norm{w}$.
\item\label{it:irred} Any vertex $v\in V$ is irreducible.
\item\label{it:unbreakable} Any vertex $v\in V$ is unbreakable.
\end{enumerate}
\end{lemma}
\begin{proof}
The statements \eqref{it:no_norm_1} and \eqref{it:norm_bound} follow immediately from Lemma~\ref{lem:key_inequality}.

Suppose that a vertex $v$ can be written as $v=x+y$ for $x,y \in \Lambda$. If we write $x=\sum c_w w$ and $y=\sum d_w w$, then since the vertices are a basis for $\Lambda$, we see that we must have $c_v\neq 0$ or $d_v\neq 0$. Without loss of generality assume that $c_v\neq 0$. Thus by \eqref{it:norm_bound}, $\norm{x}\geq \norm{v}$.
However, we also have 
\[\norm{v}=\norm{x+y}=\norm{x}+2(x\cdot y) +\norm{y},\]
showing that
\[0\leq \norm{y}\leq -2(x\cdot y).\]
Thus if $x\cdot y\geq 0$, then $\norm{y}=0$ implying that $y=0$. This shows irreducibility. If $x\cdot y=-1$, then $y\neq 0$ and $\norm{y}\leq 2$. By \eqref{it:no_norm_1} this means $\norm{y}=2$. Thus we have shown unbreakability.
\end{proof}

The following observation will also be useful.

\begin{lemma}\label{lem:same_sign} Let $\Lambda$ be a quasi-alternating lattice with vertex basis $V$. If $x = \sum_{v \in V} c_v v \in \Lambda$ is irreducible, then we have $c_v \geq 0$ for all $v$ or $c_v\leq 0$ for all $v$. 
\end{lemma}

\begin{proof}
Let $P = \{v \in V : c_v > 0\}$ and $N = \{v \in V : c_v < 0\}$  and let $w_+ = \sum_{v \in P} c_v v$ and $w_- = \sum_{v \in N} c_v v$. We have $x=w_+ + w_-$ and $w_+ \cdot w_- \geq 0$. Since $x$ is irreducible this implies that $x=w_+$ or $x=w_-$, proving that the $c_v$ must all have the same sign, as required.
\end{proof}

\section{Changemaker lattices}\label{sec:CM_lattices}
In this section we recall the changemaker theorem and the properties of changemaker lattices. The changemaker theorem was first developed by Greene for integer surgeries in his work on the lens space realization problem \cite{GreeneLRP} and the cabling conjecture \cite{greene2010space}, and for half-integer surgeries in his work on 3-braid knots with unknotting number one \cite{Greene3Braid}. It was extended to general non-integer slopes by Gibbons \cite{gibbons2013deficiency}. A proof of the changemaker theorem at the level of generality stated here can be found in the second author's thesis \cite{mccoy2016thesis}.

The changemaker theorems are obstructions to manifolds arising by positive surgery and bounding sharp negative definite manifolds.  Recall that given a negative-definite manifold $X$ with $\partial X=Y$ equipped with a \spinc -structure $\spincs$ which restricts to $\spinct$ on $Y$, there is an upper bound \cite{Ozsvath03Absolutely}:
\begin{equation}\label{eq:d_inv_bound}
d(Y,\spinct)\geq \frac{c_1(\spincs)^2 + b_2(X)}{4},
\end{equation}
where $d(Y,\spinct)$ denotes the $d$-invariant from Heegaard Floer homology. A sharp manifold is one for which \eqref{eq:d_inv_bound} is sufficient to determine all $d$-invariants on the boundary.
\begin{mydef}\label{def:sharp}
A negative definite manifold $X$ with boundary $Y$ is sharp, if for every $\spinct \in \spinc(Y)$ there is $\spincs\in \spinc(X)$ such that $\spincs$ restricts to $\spinct$ and $\spincs$ attains equality in \eqref{eq:d_inv_bound}, that is:
\[d(Y,\spinct)= \frac{c_1(\spincs)^2 + b_2(X)}{4}.\]
\end{mydef}

\begin{mydef}\label{def:CM_condition} We say that a tuple of increasing positive integers $(\sigma_1, \dots, \sigma_t)$ satisfies the {\em changemaker condition}, if for every
\[1\leq n \leq \sigma_1 + \dots + \sigma_t,\]
there is $A\subseteq \{ 1, \dots , t\}$ such that $n=\sum_{i\in A} \sigma_i$.
\end{mydef}
The changemaker has the following equivalent formulation which will sometimes be useful.
\begin{prop}[Brown, \cite{brown1961note}]\label{prop:CMcondition}
A tuple $(\sigma_1, \dots , \sigma_t)$ of increasing positive integers satisfy the changemaker condition if and only if
\[\sigma_1 = 1, \quad\text{and}\quad \sigma_i \leq \sigma_1 + \dots + \sigma_{i-1} +1,\text{ for } 1<i\leq t.\]
\end{prop}
The key definition we will need is that of a changemaker lattice.
\begin{mydef}\label{def:CMlattice}
Let $p/q>0$ be given by the continued fraction,
\[p/q=[a_0, \dots, a_l]^-=
a_0 -
    \cfrac{1}{a_1
        - \cfrac{1}{\ddots
            - \cfrac{1}{a_l} } },\]
where $a_0\geq 1$ and $a_i\geq 2$ for $i\geq 1$.
Suppose further that there is  $\{f_0, \dots f_s, e_1, \dots, e_t\}$, an orthonormal basis of $\Z^{t+s+1}$, where $s=\sum_{i=1}^l (a_i-1)$. Let $w_0, \dots, w_l \in \Z^{s+t+1}$ be such that:
\begin{enumerate}[(I)]
\item $w_0$ has norm $\norm{w_0}=a_0$ and takes the form
\[
w_0=
\begin{cases}
\sigma_1 e_1 + \dots + \sigma_t e_t &\text{if $l=0$}\\
f_0 + \sigma_1 e_1 + \dots + \sigma_t e_t &\text{if $l>0$,}
\end{cases}
\]
where $(\sigma_1, \dots, \sigma_t)$ is a tuple satisfying the changemaker condition,
\item for $k\geq 1$,
\[w_k=-f_{\alpha_{k-1}} + f_{\alpha_{k-1}+1}+\dots + f_{\alpha_{k}},\]
where
$\alpha_0=0$ and $\alpha_k=\sum_{i=1}^k (a_i-1)$.
\end{enumerate}
Then we say that the orthogonal complement
\[L=\langle w_0, \dots, w_l \rangle^\bot \subseteq \Z^{s+t+1}\]
is a {\em $p/q$-changemaker lattice}.

Moreover, we say that the $\sigma_i$ are the {\em changemaker coefficients} of $L$ and that the $\sigma_i$ satisfying $\sigma_i>1$ are the {\em stable coefficients} of $L$.
\end{mydef}
Some remarks on this definition are in order.
\begin{remark} The following observations will be useful.
\begin{enumerate}
\item As the $\alpha_i$ are defined so that $\alpha_{i}-\alpha_{i-1}=a_i-1$, the $w_i$ satisfy
\[
w_i\cdot w_j =
\begin{cases}
a_i &\text{if $i=j$,}\\
-1 &\text{if $|i-j|=1$,}\\
0 &\text{if $|i-j|>1$.}
\end{cases}
\]
\item By definition, we have $\alpha_l=s$. Thus for every $0\leq j\leq s$ there is $w_k$ with $w_k\cdot f_j=1$. As $w_0\cdot e_i>0$, for every $1\leq i\leq t$, this shows that there are no vectors of norm one in a changemaker lattice.
\item A $p/q$-changemaker lattice is determined up to isomorphism by its stable coefficients. Given the stable coefficients, the remaining changemaker coefficients are all equal to one and the number of remaining coefficients is determined by the requirement that $a_0=\norm{w_0}=\lceil p/q \rceil$. All other $w_i$ are determined by the continued fraction expansion for $p/q$.
\end{enumerate}
\end{remark}

We are now ready to state the changemaker surgery obstruction.
\begin{thm}[Theorem~2.1, \cite{mccoy2016thesis}]\label{thm:CM}
Let $K\subseteq S^3$ such that $S_{p/q}^3(K)$ bounds a sharp manifold $X$ for $p/q>0$. Then the intersection form $Q_X$ satisfies
\[-Q_X\cong L \oplus \Z^S,\]
where $S\geq 0$ is an integer and
\[L=\langle w_0, \dots, w_l \rangle^\bot\subseteq \Z^{s+t+1},\]
is a $p/q$-changemaker lattice such that for all $0\leq i \leq n/2$,
\begin{equation}\label{eq:w0formula2}
8V_{i} = \min_{\substack{ |c\cdot w_0|= n-2i \\ c \in \Char(\Z^{s+t+1})}} \norm{c} - (s+t+1),
\end{equation}
where $n=\lceil p/q \rceil$.\qed
\end{thm}
Here the $V_i$ are a non-increasing sequence of non-negative integers that are determined by the knot Floer complex $CFK^\infty$ of $K$.
\begin{remark}
It is clear from \eqref{eq:w0formula2} that the vector $w_0$ determines the $V_i$. It turns out that the sequence of $V_i$ along with equation \eqref{eq:w0formula2} is sufficient to determine the stable coefficients of $w_0$ \cite{mccoy2014bounds}. In particular, this means that the intersection form $Q_X$ is determined by the knot, the surgery slope $p/q$ and the second Betti number of $X$.
\end{remark}
In the case where $K$ is an $L$-space knot\footnote{A knot with positive $L$-space surgeries} the $V_i$ can be computed from the Alexander polynomial. For an $L$-space knot we may write its Alexander polynomial in the form
\[\Delta_K(t) = a_0 + \sum_{i=1}^g a_i(t^i+t^{-i}),\]
where $g=g(K)$ is the genus of $K$ and the non-zero values of the $a_i$ alternate in sign and take values $a_i=\pm 1$. We also assume that $\Delta_K(1)=1$. With these conventions, we define the torsion coefficients of $\Delta_K(t)$ to be
\[t_i(K)= \sum_{j\geq 1} ja_{|i|+j}.\]
For $K$ an $L$-space knot we have that $V_i=t_i(K)$.
   
\begin{remark}\label{rem:Alexander_recovery}
The torsion coefficients are sufficient to determine the Alexander polynomial. For $j\geq 1$, we can recover $a_j$ by the relation
\[a_j=t_{j-1}(K)-2t_j(K)+t_{j+1}(K).\]
Since we are normalizing so that $\Delta_K(1)=1$, this is also sufficient to recover $a_0$.
\end{remark}

When applied to Seifert fibered surgeries, Theorem~\ref{thm:CM} yields the following.
\begin{lemma}\label{lem:CM_thm_plumbings}
Let $Y=S^2(2; \frac{p_1}{q_1},\frac{p_2}{q_2},\frac{p_3}{q_3})$ be a Seifert fibered space bounding positive-definite plumbed 4-manifold $X_\Gamma$ such that $Y\cong S_{-p/q}^3(K)$ for some $K\subseteq S^3$ and $p/q>0$. Then $Y$ is an $L$-space and $Q_\Gamma \cong L$, where $L$ is the $p/q$-changemaker lattice determined by the Alexander polynomial of $\Delta_K(t)$.   
\end{lemma}
\begin{proof}
Since $Y$ arises by surgery of a negative slope, $Y$ bounds a negative definite manifold $W$ with $H_1(W;\Z) = 0$. Combined with the positive definite plumbing, this shows that $Y$ satisfies condition \eqref{it:neg_def} of Lemma~\ref{lem:QAchar}. Consequently, $Y$ satisfies the other conditions of Lemma~\ref{lem:QAchar}. This shows that $Y$ the double branched cover of a quasi-alternating link and, consequently, is an $L$-space.

Reversing orientations shows that $-Y\cong S^3_{p/q}(\overline{K})$. 
Ozsv{\'a}th and Szab{\'o} have shown that the negative definite plumbing $-X_\Gamma$ is a sharp 4-manifold \cite[Corollary~1.5]{Ozsvath03plumbed}. Since the intersection form of $-X_\Gamma$ is isomorphic to $-Q_\Gamma$, Theorem~\ref{thm:CM} applies to show that $Q_\Gamma$ is isomorphic to $L \oplus \Z^S$ for some $S\geq 0$, where $L$ is the $p/q$-changemaker lattice whose stable coefficients are determined by the Alexander polynomial of $K$. However, since $Y$ satisfies the conditions of Lemma~\ref{lem:QAchar}, the results of Lemma~\ref{lem:key_properties} apply to $Q_\Gamma$. This shows in particular that $Q_\Gamma$ contains no vectors of norm one and hence that $S=0$, as required.
\end{proof}

\subsection{Standard bases}\label{sec:standard_bases}
Having stated the changemaker surgery obstruction, we now discuss the properties of changemaker lattices that will be required. We begin first by constructing a basis for a $p/q$-changemaker lattice. See also \cite{mccoy2014noninteger}, \cite{mccoy2016thesis}. Let
\[L=\langle w_0, \dots, w_l \rangle^\bot \subseteq \Z^{s+t+1}\]
be a $p/q$-changemaker lattice for $\frac{p}{q}=n-\frac{r}{q}$ for $n>1$ and $1\leq r <q$. Let
\[w_0=f_0+\sigma_1 e_1 + \dots + \sigma_t e_t\]
and $0=\alpha_0< \dots< \alpha_l=s$ be as in the definition of $L$. Consider the set
\[M=\{0, \dots, s\}\setminus \{\alpha_1, \dots, \alpha_{l-1}\}.\]
Write $M$ as
\[M=\{\beta_0, \dots, \beta_m\},\]
where the $\beta_i$ are ordered to be increasing. Notice that $\beta_0=0$ and $\beta_m=\alpha_l=s$. For $0\leq k<m$ define
\[\mu_k=
\begin{cases}
f_0+\dots + f_{\beta_1} &\text{if $k=0$}\\
-f_{\beta_{k}}+f_{\beta_{k}+1}+\dots + f_{\beta_{k+1}} &\text{if $k>0$.}
\end{cases}
\]
These are constructed so that $\mu_k\in L$ for $k>0$. By construction the $\mu_i$ pair as follows:
\begin{equation}\label{eq:mu_i_pairings}
\mu_i \cdot \mu_j =
\begin{cases}
\norm{\mu_i} &\text{if $i=j$}\\
-1 			&\text{if $|i-j|=1$}\\
0			&\text{if $|i-j|>1$.}
\end{cases}
\end{equation}
In particular this means for any $0\leq a \leq b \leq m$ the following holds:
\begin{equation}\label{eq:norm_computation}
\norm{\mu_a + \dots + \mu_b}=2+ \sum_{i=a}^b (\norm{\mu_i}-2).
\end{equation}

It will also be useful to note that the $\mu_i$ are determined by $r/q$ by the following continued fraction identity.
\begin{lemma}[Lemma~4.8, \cite{mccoy2016thesis}]\label{lem:mui_cont_frac}
The $\mu_i$ satisfy
\[[\norm{\mu_0}, \dots, \norm{\mu_m}]^- = \frac{q}{q-r}\]
\qed
\end{lemma}

\begin{remark}\label{rem:mu_i_useful_cases}
Of particular interest will be the cases where $\frac{p}{q}=n-\frac{1}{q}$ and $\frac{p}{q}=n-\frac{q-1}{q}$. In these cases Lemma~\ref{lem:mui_cont_frac} says the following:
\begin{enumerate}[(i)]
\item If $\frac{p}{q}=n-\frac{1}{q}$, then $m=q-2$ and
\[\norm{\mu_0}=\dots = \norm{\mu_{q-2}}=2.\]
\item If $\frac{p}{q}=n-\frac{q-1}{q}$, then there is just $\mu_0$ and it satisfies $\norm{\mu_0}=q$.
\end{enumerate}
\end{remark}

For $1\leq k\leq t$, we say that $\sigma_k$ is {\em tight} if
\[\sigma_k=1+\sigma_1+\dots + \sigma_{k-1}.\]
If $\sigma_k$ is not tight, then Proposition~\ref{prop:CMcondition} shows that there is a subset $A \subseteq \{1, \dots, k-1\}$ such that $\sigma_k=\sum_{i\in A}\sigma_i$. For each $k$, let $A_k$ denote the maximal such subset with respect to the lexicographical ordering on subsets of $\{1, \dots, k-1\}$.
Define $\nu_k$ by
\[
\nu_k=
\begin{cases}
-e_k+e_{k-1}+\dots + e_{1}+\mu_0 & \text{if $\sigma_k$ is tight}\\
-e_k+\sum_{i\in A_k} e_i &\text{otherwise.}
\end{cases}
\]
Note that in any changemaker lattice $\sigma_1=1$ is always tight and we have $\nu_1=-e_1+\mu_0$. 
We say that a standard basis element $\nu_k$ is {\em gapless} if it takes the form\footnote{Such elements were called {\em just right} by Greene.}
\[\nu_k=-e_k+e_{k-1} + \dots + e_l\]
for some $l<k$.

\begin{remark}\label{rem:Standard_basis_facts} The lexicographical maximality condition on $A_k$ has the following useful consequences.
\begin{enumerate}[(i)]
\item\label{it:k-1_in_A_k} For $k>1$, we always have $\nu_k \cdot e_{k-1}=1$. When $\sigma_k$ is tight this is by definition. When $\sigma_k$ is not tight, Proposition~\ref{prop:CMcondition} shows that we can construct the set $A_k$ by a ``greedy algorithm''. Under such an algorithm, $k-1$ is the always the first element to be included in $A_k$.
\item\label{it:gapless_converse} If $v\in L$ takes the form
\[v=-e_k+e_{k-1} + \dots + e_l,\] 
then $v=\nu_k$ is necessarily a {\em gapless} standard basis vector.
\end{enumerate}
\end{remark}
We say that
\[S=\{\nu_1, \dots, \nu_t, \mu_1, \dots, \mu_m \}\]
is the {\em standard basis} for $L$. The standard basis is, in fact, a basis for $L$.
\begin{lemma}[Proposition~4.9, \cite{mccoy2016thesis}]\label{lem:standard_basis}
The standard basis $S$ is a basis for $L$. \qed
\end{lemma}

Recall that the notions of irreducibility and unbreakability are given in Definition~\ref{def:irred/unbreak}.
\begin{lemma}[Lemma~4.13, \cite{mccoy2016thesis}]\label{lem:standard_basis_irred}
Every element $v\in S$ is irreducible. \qed
\end{lemma}
We will also require the following structure result on certain irreducible and unbreakable elements of $L$. It is an extension of Lemma~4.16 and Lemma~4.17 of \cite{mccoy2016thesis}. 
\begin{lemma}\label{lem:irred_frac_part2}
Let $v\in L$ be irreducible and unbreakable with $v\cdot f_i \ne 0$ for some $i$.
\begin{enumerate}[(i)]
\item If $v\cdot f_0=0$, then $v$ takes the
form 
\[\pm v= \mu_a+ \dots +\mu_b,\]
where there is at most one $c$ in the range $a\leq c \leq  b$ with $\norm{\mu_c}>2$.
\item If $v\cdot f_0\ne 0$, then $v$ takes the form
\[\pm v= -e_g + e_{k-1} + \dots +e_1 + \mu_0+ \dots + \mu_b,\]
where $\sigma_k$ is tight, $\sigma_g=\sigma_k$ and $\norm{\mu_i}=2$ for $1\leq i \leq b$.
\end{enumerate} 
\end{lemma}
\begin{proof}
Since $v$ is irreducible, it follows from Lemma~4.16 and Lemma~4.17 of \cite{mccoy2016thesis} that if $v\cdot f_0=0$, then $v$ the form
\[\pm v = \mu_a+ \dots +\mu_b.\]
for $1\leq a \leq b \leq m$. We claim the unbreakability of $v$ implies that there is at most one $a\leq c\leq b$ with $\norm{\mu_c}>2$. Take $c$ to be minimal such that $\norm{\mu_c}>2$. If $c<b$, then we have 
\[(\mu_a+\dots + \mu_c) \cdot (\mu_{c+1} + \dots + \mu_b)=-1.\]
Thus, the unbreakability of $v$ implies that we must have $\norm{\mu_{c+1} + \dots + \mu_b} =2$, and hence by \eqref{eq:norm_computation} that $\norm{\mu_{c+1}}= \dots = \norm{\mu_b}=2$. 
Similarly, if $a < c$ then \[(\mu_a+\dots + \mu_{c-1}) \cdot (\mu_{c} + \dots + \mu_b)=-1,\] implying that $\norm{\mu_a} = \cdots = \norm{\mu_{c-1}} = 2$, as required.

Now suppose that $v\cdot f_0\neq 0$. In this case Lemma~4.16 and Lemma~4.17 of \cite{mccoy2016thesis} imply that $v$ takes the form
\[\pm v= x_I + x_F\]
where $x_I \neq 0$, $x_I \cdot f_i=0$ for all $i$ and $x_F$ takes the form
\[x_F=\mu_0+ \dots + \mu_b.\]
Since $\mu_1, \dots, \mu_b$ are in $L$, we have $x_I + \mu_0\in L$. We also have that $\norm{x_I+\mu_0}>\norm{\mu_0}\geq 2$. So by applying the unbreakability condition to $(x_I + \mu_0)\cdot (\mu_1 + \dots + \mu_b)=-1$ we obtain that
\[\norm{\mu_1+\dots + \mu_b}=2.\]
Using \eqref{eq:norm_computation}, this implies that 
\[\norm{\mu_1}= \dots= \norm{\mu_b}=2,\] 
as required.

Now we study the structure of $x_I$. Let $k\geq 1$ be minimal such that $x_I\cdot e_k\leq 0$. By Proposition~\ref{prop:CMcondition}, there is a subset $B\subseteq \{1, \dots, k-1\}$ such that
\[\sigma_k-1 = \sum_{i\in B} \sigma_i.\]
Thus we can consider
\[z=-e_k+ \sum_{i\in B} e_i + x_F \in L.\]
Note that by assumption, we have $x_I\cdot e_i\geq 1$ for all $i<k$ and hence for all $i \in B$. Thus we obtain the bound
\begin{align}\begin{split}
(x_I+ x_F-z)\cdot z 
&= -(x_I\cdot e_k +1) + \sum_{i\in B} (x_I\cdot e_i-1) \\
&\geq -x_I\cdot e_k -1\geq -1.
\end{split}\end{align}
Thus by the assumption of irreducibility we have that
\[(x_I+x_F-z)\cdot z =
\begin{cases}
0 &\text{if $z=x_I+x_F$}\\
-1 &\text{otherwise.}
\end{cases}
\]
Suppose first that $z=x_I+x_F$. Since $k$ was chosen to minimal such that $x_I\cdot e_k\leq 0$, we have
\[
x_I+x_F= -e_k+e_{k-1}+ \dots +e_1+ \mu_0 + \dots +\mu_b,
\]
which is in the required form.
Thus we can assume that 
\[(x_I+x_F-z)\cdot z=-1\]
which can only occur if $x_I\cdot e_k=0$. Since $\norm{z}>2$, it follows from the indecomposability condition that $x_I+x_F-z$ has norm two. We have $(x_I+x_F-z)\cdot e_k = -(z\cdot e_k)=1$. Thus $x_I+x_F-z$ takes the form 
\[x_I+x_F-z=e_k -\varepsilon e_g\]
for some $g\neq k$ and some $\varepsilon \in \{\pm 1\}$. The fact that $(x_I+x_F-z)\cdot w_0 = 0$ shows that $\sigma_g=\sigma_k$ and $\varepsilon=1$.
Thus we have that
\[x_I+x_F= z+e_k-e_g\]
for some $g$ with $\sigma_g=\sigma_k$. Since $k$ is minimal with $v\cdot e_k\leq 0$, it follows that $g>k$ and
\[x_I+x_F= -e_g + e_{k-1} + \dots +e_1 + \mu_0+ \dots + \mu_b,\]
as required.
\end{proof}

\begin{remark}
When rewritten in terms of the orthonormal basis for $\Z^{s+t+1}$ the two types of vector arising in the previous lemma are
\[
\mu_a+ \dots +\mu_b = -f_{\beta_a}+ f_{\beta_c+1}+ \dots + f_{\beta_{c+1}-1} + f_{\beta_{b+1}},
\]
where, if it exists, $c$ is the unique $c$ in the range $a\leq c\leq b$ with $\norm{\mu_c}>2$
and 
\[-e_g + e_{k-1} + \dots +e_1 + f_0+ \dots + f_{\beta_1-1} + f_{\beta_{b+1}}.\]
\end{remark}

We end with a final useful observation.

\begin{remark}\label{rem:relabel}
There is a certain redundancy in the choice of indexing of the $f_0, \dots, f_s$ and $e_1, \dots, e_t$. Whenever $\sigma_a=\sigma_b$ for $a\ne b$ (equivalently if $e_a-e_b\in L$), then we can reindex the $e_i$ to exchange $e_a$ and $e_b$. Similarly given $f_a$ and $f_b$ such that $f_a-f_b\in L\setminus \{0\}$, then we can exchange $f_a$ and $f_b$. More formally, this is the observation that automorphism of $\Z^{s+t+1}$ exchanging $e_a$ and $e_b$ or $f_a$ and $f_b$ preserves $L$ as subset of $\Z^{s+t+1}$. We will make frequent use of such relabellings in Section~\ref{sec:non_int_surgeries}.
\end{remark}

\section{Analysis for the $e=2$ case}
\label{sec:non_int_surgeries}
Although the formal proof of Theorem~\ref{thm:nonint} is stated in Section~\ref{sec:proof_of_main}, this section contains the analysis necessary to prove the Theorem~\ref{thm:nonint} for $e=2$. The section culminates in Lemma~\ref{lem:lattices_to_surgery} which combines with Lemma~\ref{lem:CM_thm_plumbings} to give the proof. 

Let $L$ be a $p/q$-changemaker lattice
\[
L=\langle w_0, \dots, w_l \rangle^\bot \subseteq \langle f_0, \dots, f_s, e_1, \dots, e_t \rangle = \Z^{s+t+1}
\]
for $q>1$. Suppose that $L$ is isomorphic to the intersection form of some plumbing $\Gamma$ (as in Figure~\ref{fig:star_plumbing}) with $e=2$. Let $V$ denote the image of the vertices of $\Gamma$ in $L$. In a mild abuse of notation we will simply refer to the elements of $V$ as the vertices of $\Gamma$. We seek to understand the structure of $V$ and $\Gamma$. The eventual aim is to show that if $Y$ is the Seifert fibered spaces for which $\Gamma$ is the canonical plumbing then $Y$ arises by $p/q$-surgery. In order to do this, we will take $L$ to have standard basis elements
\[\{\nu_1, \dots, \nu_t, \mu_1, \dots, \mu_m\},\]
as defined in Section~\ref{sec:standard_bases}.

Key to this section will be the observation that $\Gamma$ is quasi-alternating. Consequently the results of Section~\ref{sec:SF_plumbings} apply, showing, in particular, that the vertices are irreducible and unbreakable.
\begin{prop}\label{prop:lambda_is_QA}
The plumbing graph $\Gamma$ is quasi-alternating.
\end{prop}
\begin{proof}
  Let $A$ be the matrix representing the inclusion $L \rightarrow \Z^{s+t+1}$ with respect to the standard basis for $L$ and the orthonormal basis for $\Z^{s+t+1}$. By ordering the basis vectors appropriately $A^T$ takes the form

\[
A^T=
\begin{pmatrix}
\nu_t \cdot f_s & \dots & \nu_t\cdot f_0 & \nu_t\cdot e_1 & \dots & \nu_t\cdot e_t \\
\vdots          &        & \vdots         & \vdots         &        &\vdots \\
\nu_1 \cdot f_s & \dots & \nu_1\cdot f_0 & \nu_1\cdot e_1 & \dots & \nu_1\cdot e_t \\
\mu_1 \cdot f_s & \dots & \mu_1\cdot f_0 & \mu_1\cdot e_1 & \dots & \mu_1\cdot e_t \\
\vdots          &        & \vdots         & \vdots         &        &\vdots \\
\mu_m \cdot f_s & \dots & \mu_m\cdot f_0 & \mu_m\cdot e_1 & \dots & \mu_m\cdot e_t \\
\end{pmatrix}
\]
However by definition of the standard basis elements, this matrix is in row echelon form and the first non-zero entry in each row is $-1$. Consequently $A^T$ is surjective over the integers. This shows that Lemma~\ref{lem:QAchar}\eqref{it:embed_condition} is satisfied. Therefore, Lemma~\ref{lem:QAchar}\eqref{it:num_condition} applies to show $\Gamma$ is quasi-alternating. 
\end{proof}

Now we set about understanding the vertices of $\Gamma$ in $L$.

\begin{lemma}\label{lem:mu_i_are_vertices}
We may assume that $\mu_1, \dots, \mu_m$ are vertices.
\end{lemma}
\begin{proof}
We prove the lemma inductively by establishing that if $\mu_{k+1}, \dots, \mu_m$ are vertices, then we may further assume that $\mu_k$ is a vertex.
 
Since the vertices of $\Gamma$ span $L$, there are integers $c_v$ such that $\mu_k =\sum_{v\in \Gamma} c_v v$. For any $v$ with $c_v\ne 0$, Lemma~\ref{lem:key_properties}\eqref{it:norm_bound} shows that $\norm{v}\leq \norm{\mu_k}$. We may write each $v$ as an integer combination of the standard basis elements in a unique way. Thus we see there must be some $v$ with $c_v\neq 0$, for which $\mu_k$ appears with non-zero coefficient when $v$ is expressed as an integer combination of standard basis elements. As $v$ is irreducible and unbreakable, Lemma~\ref{lem:irred_frac_part2} combined with the fact that $\norm{v}\leq \norm{\mu_k}$ shows that $v$ takes the form
\[\pm v = \mu_a + \dots + \mu_b,\]
where $a\leq k \leq b$ and there is at most one $c$ in the range $a\leq c \leq b$ with $\norm{\mu_c}\geq 2$. If such a $c$ exists, then we have $k=c$, since $\norm{v}=\norm{\mu_c}\leq \norm{\mu_k}$. Thus we have $\norm{\mu_i}=2$ for $a \leq i<k$ and $k <i\leq b$. If $a<k$, then a relabelling of the $f_i$ (the one exchanging the roles of $f_{\beta_a}$ and $f_{\beta_k}$) allows us to assume that $a=k$.

If $k=m$, then we have shown that we can assume $\pm \mu_m$ is a vertex. So, by multiplying all vertices by $-1$ if necessary, we can assume that $\mu_m$ is a vertex. This deals with the base case of the induction.

Thus suppose that $k<m$. By the previous discussion we can assume there is a vertex $v$ of the form $v=\varepsilon (\mu_k+ \dots + \mu_b)$. One can easily calculate that 
\begin{equation}\label{eq:vdotmui}
\mu_i \cdot v =
\begin{cases}
0 				&\text{if $k<i<b$}\\
\varepsilon 	&\text{if $i=b$}\\
-\varepsilon 	&\text{if $i=b+1$}\\
0 				&\text{if $i>b+1$.}
\end{cases}
\end{equation}
Since $\mu_{k+1}, \dots, \mu_m$ form a connected chain of vertices, $v$ can pair non-trivially with at most one of them and this pairing must be $-1$. Thus it follows from \eqref{eq:vdotmui}, that we must have either have $b=k$ and $\varepsilon=1$, or $b=m$ and $\varepsilon=-1$. In the former case we must have $v=\mu_k$ as required. In the latter, we have 
\begin{equation}
v=-(\mu_{k}+ \dots + \mu_m)
\end{equation}
However in this case we have
\[
\mu_{k+1}=-f_{\beta_{k+1}}+f_{\beta_{k+2}}, \dots, \mu_m=-f_{s-1}+f_s
\]
 are all of norm two and that
\[
v=f_{\beta_{k}}-f_{\beta_{k}+1}- \dots -f_{\beta_{k+1}-1}- f_s.
\] 
Thus if we relabel the $f_i$ so as to reverse the order of the $f_{\beta_{k+1}}, \dots, f_s$, then the set of vertices $\{v, \mu_{k+1}, \dots, \mu_m\}$ becomes $\{-\mu_{k}, \dots, -\mu_m\}$. Therefore, after multiplying every vertex by -1, we may assume that we have the desired set of vertices.

This verifies the inductive step and completes the proof. 
\end{proof}

\begin{lemma}\label{lem:v_form}
Let $v$ be a vertex distinct from $\mu_1, \dots, \mu_m$ with $v\cdot f_i\ne 0$ for some $i$, then $v$ takes the form
\begin{equation}\label{eq:v+type}
v=-e_g + e_{k-1}+ \dots + e_1 + \mu_0,
\end{equation}
or
\begin{equation}\label{eq:v-type}
v=e_g - e_{k-1}-\dots - e_1 - \mu_0 - \dots  -\mu_m,
\end{equation}
where $k \le g$, and this latter case can occur only if
\[\norm{\mu_1}=\dots = \norm{\mu_m}=2.\]
\end{lemma}
\begin{proof}
Since every vertex is irreducible and unbreakable, by Lemma~\ref{lem:irred_frac_part2} we see that either $v$ is a linear combination of the $\mu_1, \dots, \mu_m$ or it has $v\cdot f_0\neq 0$. Since the vertices are linearly independent, we must have $v\cdot f_0\neq 0$. By Lemma~\ref{lem:irred_frac_part2}, we may assume that such a vertex takes the form
\begin{equation}
v=\varepsilon(-e_g + e_{k-1}+ \dots + e_1 + \mu_0 + \dots + \mu_b)
\end{equation}
for some $\varepsilon \in \{\pm 1\}$ and  $g\geq k$ with $\sigma_k=\sigma_g$ and $\sigma_k$ is tight and $\norm{\mu_1}=\dots = \norm{\mu_b}=2$. Since the $\mu_i$ form a linear chain of vertices, we see that $v$ can have non-zero pairing with at most one of them.
However, as we have the following pairings, 
\begin{equation}\label{eq:mu_i_pairing}
\mu_i \cdot v =
\begin{cases}
0 				&\text{if $0<i<b$}\\
\varepsilon 	&\text{if $i=b$}\\
-\varepsilon 	&\text{if $i=b+1$}\\
0 				&\text{if $i>b+1$},
\end{cases}
\end{equation}
we have either $\varepsilon=1$ and $b=0$, or $\varepsilon=-1$ and $b=m$. In the $\varepsilon=1$ and $b=0$ case, this puts $v$ in the form of \eqref{eq:v+type}. In the $\varepsilon=-1$ and $b=m$ case, this puts $v$ in the form of \eqref{eq:v-type}.
\end{proof}

\begin{lemma}\label{lem:nu1_is_vertex}
We may assume that $\nu_1$ is a vertex.
\end{lemma}
\begin{proof}
Expressing $\nu_1$ as a linear combination of vertices, we see that there must be a vertex $v$ with $v\cdot f_0\ne 0$, and $\norm{v}\leq \norm{\nu_1}=\norm{\mu_0} +1$ by Lemma~\ref{lem:key_properties}. We see that such a vertex must take either the form
\begin{equation}\label{eq:nu_1+form}
v=-e_g + \mu_0,
\end{equation}
coming from \eqref{eq:v+type}, or the form 
\begin{equation}\label{eq:nu_1-form}
v=e_g-\mu_0- \dots - \mu_m,
\end{equation}
coming from \eqref{eq:v-type}. In both cases $\sigma_g=\sigma_1=1$ and in the latter case $\norm{\mu_i}=2$ for $1\leq i \leq m$. By relabelling the $e_i$s, we may assume that $g=1$. Thus there is nothing further to check when $v$ take the form given in \eqref{eq:nu_1+form}. So suppose that $v$ takes the form given in \eqref{eq:nu_1-form}.
In this case, we apply an argument similar to the one at the end of the proof of Lemma~\ref{lem:mu_i_are_vertices}. We can relabel the $f_i$ so as to reverse the order of $f_{\beta_{1}}, \dots, f_s$. Under this relabelling the vertices $\mu_1, \dots, \mu_m$ become $-\mu_m, \dots, -\mu_1$ and $v$ becomes $-\nu_1$. Thus by reversing signs on all vertices, we can assume that $\nu_1, \mu_1, \dots , \mu_m$ are all vertices, as required.
\end{proof}

\begin{lemma}\label{lem:v_endtype}
If $v \not\in \{\nu_1,\mu_1, \dots, \mu_m\}$ is a vertex, then either
\begin{enumerate}[(a)]
\item $v\cdot e_1=0$ and $v\cdot f_i=0$ for all $0\leq i\leq s$ or
\item $v\cdot e_1=1$ and $v\cdot f_i=0$ for all $0\leq i\leq s$ or
\item $\frac{p}{q}=n-\frac{1}{q}$ and $v$ can be assumed to take the form
\[v=e_k-e_{k-1} - \dots - e_1 -\mu_0 - \dots - \mu_m\]
where $k>1$ and $\sigma_k$ is tight. Moreover, there is at most one vertex of type~(c).
\end{enumerate}
\end{lemma}
\begin{proof}
Let $v\neq \nu_1,\mu_1, \dots, \mu_m$ be a vertex with $v\cdot f_i\neq 0$ for some $i$. By Lemma~\ref{lem:v_form}, there are two possible forms for $v$. First assume that $v$ takes the form given in $\eqref{eq:v+type}$. In this case, we have $v\cdot \nu_1\geq \norm{\mu_0}-1>0$, which is impossible unless $v=\nu_1$. Thus $v$ must take the form given in \eqref{eq:v-type}.

If $m=0$, then 
\[
  v\cdot \nu_1=-\norm{\mu_0}-v\cdot e_1 \in \{ -\norm{\mu_0} \pm 1, -\norm{\mu_0} \}.
\]
However since $v$ and $\nu_1$ are both vertices $v\cdot \nu_1\in \{0,-1\}$. As $\norm{\mu_0}\geq 2$, this implies that $v\cdot e_1=-1$ and $\norm{\mu_0}=2$. This implies that $q=2$ and $k>1$ (cf. Remark~\ref{rem:mu_i_useful_cases}).  
If $m>0$, the argument is similar. Since $\mu_m\cdot v=-1$ and $\nu_1,\mu_1,\ldots,\mu_m$ form a linear chain of vertices, we must have $v\cdot \nu_1=0$. This implies that
\[v\cdot \nu_1 = -(\norm{\mu_0}-1) -v\cdot e_1 =0.\]
This shows that $\norm{\mu_0}=2$ and $v\cdot e_1=-1$. In either case this shows that $p/q$ takes the form $\frac{p}{q}=n-\frac{1}{q}$ (cf. Remark~\ref{rem:mu_i_useful_cases}). Since $v\cdot e_1=-1$, it follows that $k>1$.

To see that such a $v$ is necessarily unique, suppose that $v$ and $w$ are both vertices of the form given in \eqref{eq:v-type}. For such vertices we have 
\[v\cdot w \geq \norm{\mu_0+ \dots +\mu_m} -1 >0,\]
which is impossible, unless $v=w$.

Given that such a $v$ is unique and $k>1$, we see that there is no loss of generality in relabelling the $e_i$ to assume that $g=k$. This shows that $v$ can be taken in the form given by~(c).

Finally, consider the case that $v$ is a vertex $v$ with $v\cdot f_i=0$ for all $i$. Since $\nu_1$ is a vertex, we have
\[v\cdot \nu_1= -v\cdot e_1 \in \{0,-1\}.\]
This shows that $v$ is in the form described by (a) or (b), as required.
\end{proof}

\begin{figure}[!ht]
  \centerline{
    \begin{overpic}[width=0.5\textwidth]{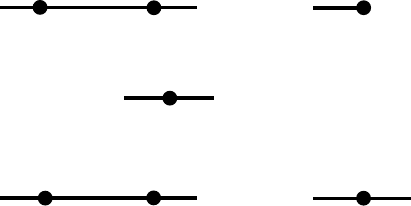}
      \put (-35,47) {\large Type~I:}
      \put (9,42) {$\nu_1$}
      \put (35,42) {$\mu_1$}
      \put (57,48.5) {\Large $\dots$}
      \put (86,42) {$\mu_m$}
      \put (-35,25) {\large Type~II:}
      \put (40,20) {$\nu_1$}
      \put (-35,0) {\large Type~III:}
      \put (9,-4) {$\nu_1$}
      \put (36,-4) {$\mu_1$}
      \put (57,2.4) {\Large $\dots$}
      \put (85,-4) {$\mu_{q-2}$}
    \end{overpic}
  }
\caption{The three types.}
\label{fig:3_types}
\end{figure}
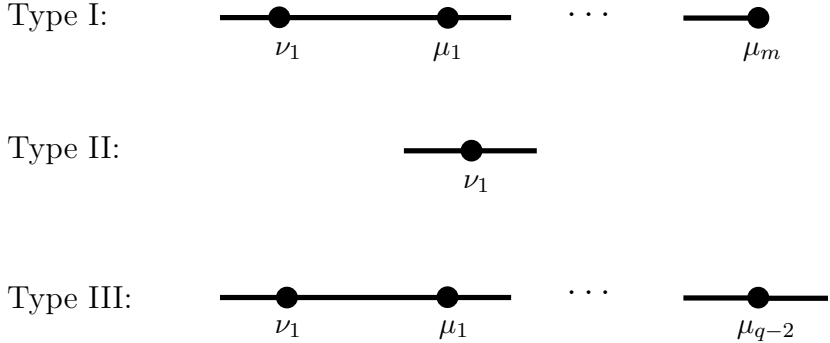

Given a vertex $v\ne \nu_1, \mu_1, \dots, \mu_m$, we refer to it as being of type (a), (b) or (c) if it satisfies conditions (a), (b) or (c) from Lemma~\ref{lem:v_endtype}, respectively. This allows us to show that the vertex set satisfies the following trichotomy.

\begin{lemma}\label{lem:vertex_set_trichotomy}
The vertex set takes one and only one of the following forms:
\begin{enumerate}[I.]
\item There are no type~(c) vertices and $\nu_1$ is adjacent to a single vertex of type~(b),
\item $\frac{p}{q}=n-\frac{q-1}{q}$ and $\nu_1$ is adjacent to two vertices of type~(b),
\item or $\frac{p}{q}=n-\frac{1}{q}$ and there is a unique vertex of type~(c) and at most one vertex of type~(b).
\end{enumerate}
\end{lemma}
\begin{proof}
As we are assuming that the central vertex is of norm two and $\norm{\nu_1}>2$, we see that $\nu_1$ is not the central vertex of $\Gamma$. Thus $\nu_1$ pairs with at most two vertices in the graph. Since a vertex of type~(b) is always adjacent to $\nu_1$, this shows there are at most two vertices of type~(b). 

Suppose that the vertex set contains two vertices of type~(b). We will show that the vertex set takes the Type~II form. Since both vertices of type~(b) pair with $\nu_1$, the vertex $\mu_1$ cannot exist and hence $p/q$ takes the form $\frac{p}{q}=n-\frac{q-1}{q}$ by Remark~\ref{rem:mu_i_useful_cases}. 
If the vertex set further contains a vertex of type~(c), since such a vertex would be a third vertex adjacent to $\nu_1$. Thus the vertex set takes the Type~II form. 

Now suppose that the vertex set contains a vertex of type~(c). We will show that the vertex set takes the Type~III form. By the argument in the previous paragraph, the vertex set contains at most one vertex of type~(b). By Lemma~\ref{lem:v_endtype} the existence of a vertex of type~(c), shows that $p/q$ takes the form $\frac{p}{q}=n-\frac{1}{q}$. This puts the vertex in the form Type~III.

Finally suppose that the vertex set contains no vertex of type~(c) and at most one vertex of type~(b). Note that if there is no vertex of type~(b) then the graph $\Gamma$ would have a connected component consisting of the linear chain $\nu_1, \mu_1, \dots, \mu_m$, which is incompatible with our assumptions on $\Gamma$. Thus the vertex set contains a unique vertex of type~(b) and is hence of the form given by Type~I, as required.
\end{proof}
The local structure of each these three types is shown in Figure~\ref{fig:3_types}.
It turns out that a Type~I vertex set corresponds to surgery on a torus knot. Type~II and Type~III vertex sets both correspond to surgery on a cable of a torus knot.

\subsection{Type~I and Type~II}
Now that we understand vertices pairing nontrivially with the $f_i$, we turn our attention to the remaining vertices. In the case where there are no vertices of type~(c), these vertices can be taken to be exactly the standard basis elements. 
\begin{lemma}\label{lem:structure}
If the vertex set of $\Gamma$ is of Type~I or Type~II, then we can assume that the vertices are the standard basis elements and that these are all gapless.
\end{lemma}
\begin{proof}
We prove inductively that we can take the vertices to be standard basis elements. By Lemma~\ref{lem:mu_i_are_vertices} and Lemma~\ref{lem:nu1_is_vertex}, we can assume that $\nu_1, \mu_1,\dots, \mu_m$ are vertices. This is the base case.

Now assume that $\mu_1,\dots, \mu_m, \nu_1, \dots, \nu_k$ are all vertices.
\begin{claim}
Suppose that $v$ is a vertex, which is not one of the $\mu_1,\dots, \mu_m, \nu_1, \dots, \nu_k$. Then $v$ has the following properties:
\begin{enumerate}[(i)]
\item $v\cdot f_i=0$ for all $i$,
\item $v\cdot e_i\geq 0$ for $1\leq i \leq k$ and
\item if $v\cdot e_j>0$ for some $j<k$, then $v\cdot e_i>0$ for all $j\leq i\leq k$.
\end{enumerate}
\end{claim}
\begin{proof}[Proof of Claim.]
 By the assumption that there are no type~(c) vertices, we have $v\cdot f_i=0$ for all $i$. Now suppose that $v\cdot e_i \neq 0$ for some $1\leq i \leq k$. Let $l\geq 1$ be minimal such that $v\cdot e_l \neq 0$. In this case we have $v\cdot \nu_l= -v\cdot e_l$. As $v$ and $\nu_l$ are both vertices then this shows that $v\cdot e_l= 1$. Now let $g>l$ be minimal such that $v\cdot e_g \leq 0$.
By Remark~\ref{rem:Standard_basis_facts}, we have that $\nu_{g}\cdot e_{g-1}=1$. Therefore we see that 
\[v\cdot \nu_g\geq -v\cdot e_g +v\cdot e_{g-1}>0.\] From this we conclude that  either $v=\nu_g$ or $\nu_g$ is not a vertex. In either case this implies $g>k$. This gives $(ii)$ and $(iii)$.
\end{proof}
Let $v_1, \dots, v_N$, be the vertices which are not already known to be standard basis elements. The preceding claim shows that each $v_j$ can be written as $v_j=v_j'+v_j^+$, where
\[v_j'\cdot e_i=0 \quad\text{for $i\leq k$}\]
and
\[
v_j^+\cdot e_i\geq 0 
\quad\text{for $i\leq k$, and }
v_j^+\cdot e_i= 0 \quad\text{for $i> k$.}
\]

Now consider $\nu_{k+1}$. There are integer $\alpha_i$ and $\beta_j$ such that
\begin{equation}\label{eq:nu_k+1_sum}
\nu_{k+1}=\sum_{i=1}^k \alpha_i \nu_i + \sum_{j=1}^N \beta_j v_j
\end{equation}
{\em A priori} one might expect the $\mu_i$ to appear in this sum. However it follows from considering the pairing with the $f_i$ that there is no need to include them. By construction of the standard basis vectors $\nu_{k+1} \cdot f_i$ can be non-zero only if $i\leq \beta_1$. If there were $\mu_i$ appearing in the sum \eqref{eq:nu_k+1_sum}, then we would have $\nu_k \cdot f_i\neq 0$ for some $i>\beta_1$, contradicting this.

Since $\nu_{k+1}$ is irreducible, Lemma~\ref{lem:same_sign} shows that all non-zero $\alpha_i$ and $\beta_j$ must have the same sign.

Now if we write $\nu_{k+1}$ in the form $\nu_{k+1}=-e_{k+1}+ \nu^+$, then \eqref{eq:nu_k+1_sum} yields
\begin{equation}\label{eq:nu+_sum}
\nu^+=\sum_{i=1}^k \alpha_i \nu_i + \sum_{j=1}^N \beta_j v^+_j.
\end{equation}
By taking the pairing of \eqref{eq:nu+_sum} with $w_0$ and observing that, by construction, $w_0\cdot v_j=0$ for all $j$, we obtain
\begin{equation}\label{eq:sigma_k_dot}
\sigma_{k+1}=\nu^+\cdot w_0=\sum_{j=1}^N \beta_j (v^+_j \cdot w_0).
\end{equation}
Since $\sigma_{k+1}>0$ and $v^+_j \cdot w_0\geq 0$ for all $j$, this shows that the $\alpha_i$ and $\beta_j$ must all be non-negative.

Let $\lone{x}$ denote the $\ell_1$-norm
\[\lone{x}=\sum_{i=1}^{t} |x\cdot e_i| + \sum_{j=0}^s |x\cdot f_j|.\]
Since the coefficients of $\nu^+$ are equal to 0 or 1, we have $\lone{\nu^{+}}=\norm{\nu^{+}}$. However by writing $\nu^+$ as a sum in \eqref{eq:nu+_sum} and computing $\lone{\nu^+}$ we obtain
\begin{equation}\label{eq:nu+_norm}
\norm{\nu^{+}}=\sum_{i=1}^k \alpha_i (\lone{\nu_i}-2) + \sum_{j=1}^N\beta_i \lone{v_i^{+}},
\end{equation}
where the $\lone{\nu_i}-2$ terms come from the fact that $\nu_i\cdot e_i=-1$ and $\nu_i\cdot e_j\geq 0$ for $j\neq i$.

By the inequality in Lemma~\ref{lem:key_inequality} we have the bound:
\begin{align}\begin{split}\label{eq:applied_ineq}
\norm{\nu_{k+1}} &=\norm{\nu^{+}}+1\\
&\geq 2+ \sum_{i=1}^k \alpha_i (\norm{\nu_i}-2) + \sum_{j=1}^N\beta_j(\norm{v_j}-2)\\
&=2+ \sum_{i=1}^k \alpha_i (\lone{\nu_i}-2) + \sum_{j=1}^N \beta_j(\norm{v_j^{+}}+ \norm{v_j'}-2) \\
&=\norm{\nu^{+}} +2 + \sum_{j=1}^N \beta_j(\norm{v^+_j}-\lone{v^+_j}+\norm{v_j'}-2),
\end{split}\end{align}
where \eqref{eq:nu+_norm} was used to obtain the last line.
Comparing the first and last lines in \eqref{eq:applied_ineq} shows that
\begin{equation}\label{eq:v'_bounds}
 \sum_{j=1}^N \beta_i(\norm{v^+_i}-\lone{v^+_i}+\norm{v_i'}-2)\leq -1.
\end{equation}
Since there must be at least one negative summand on the left hand side of \eqref{eq:v'_bounds}, we can assume that
\[ 
\norm{v^+_1}-\lone{v^+_1}+\norm{v_1'}\leq 1\quad\text{and}\quad \beta_1\geq 1.
\]
Since $\norm{v_1'}\geq 1$ and $\norm{v^+_1}\geq\lone{v^+_1}$, we must have $\norm{v_1'}= 1$ and $\norm{v^+_1}=\lone{v^+_1}$
However $\norm{v^+_1}=\lone{v^+_1}$ only if $v^+_1\cdot e_j\in \{0,\pm 1\}$ for all $j$. By the restrictions on $v_1$ proven in the claim at the start of the proof, this shows that $v_1$ takes the form
\begin{equation}\label{eq:v_1_form}
v_1=-e_g+e_k+ \dots + e_l
\end{equation}
for some $g>k$ and $l\leq k$.

Since $g>k$ we have that $\sigma_g \geq \sigma_{k+1}$. On the other hand, the condition $v_1\cdot w_0=0$ implies that $v_1^+\cdot w_0=e_g\cdot w_0=\sigma_g$. Furthermore, by computing as in \eqref{eq:sigma_k_dot} and using the fact that $\beta_1\geq 1$ and that $\beta_j (v^+_j \cdot w_0)\geq 0$ for each $j$ shows that 
\[\sigma_{k+1}= \sum_{i=1}^N \beta_i (v^+_j \cdot w_0)\geq v^+_1 \cdot w_0.\]
Thus we have $\sigma_{k+1}=\sigma_g$. By relabelling we can assume that $v_1=-e_{k+1}+e_k+ \dots + e_l$. As mentioned in Remark~\ref{rem:Standard_basis_facts} it follows that $\nu_{k+1}=v_1$ is a gapless standard basis vector. Thus we have shown we may assume that $\nu_{k+1}$ is vertex. This completes the inductive step of the proof. 
\end{proof}
This has several useful consequences.
\begin{remark}\label{rem:structure_consequences} Suppose that $\Gamma$ is a plumbing whose intersection form is isomorphic to a $p/q$-changemaker lattice $L$ with Type~I or Type~II vertex set.
\begin{enumerate}[(i)] 
\item\label{it:L_determines_Gamma} Since the vertices can be taken to be standard basis elements of $L$, the plumbing graph $\Gamma$ is completely determined by $L$. 
\item\label{it:L_has_no_tight} $L$ can have no tight standard basis elements except $\nu_1$. Since a Type~III vertex set implies the existence of a tight standard basis element, this shows that the type of vertex set is intrinsic to the lattice $L$ rather than the plumbing $\Gamma$ or the choice of vertex set.
\item\label{it:local_struc} Since there can be no tight standard basis elements we have 
  $\nu_2=-e_2+e_1$
  as one type~(b) vertex. In the Type~II case, the other type~(b) vertex must take the form
\[-e_g+e_{g-1}+ \dots + e_1,\]
for some $g>1$. This shows that $\Gamma$ takes the form shown in Figure~\ref{fig:further_structure}. Recall that in the Type~II case, there is no vertex $\mu_1$ (cf. Remark~\ref{rem:mu_i_useful_cases}(ii)).
\end{enumerate}
\end{remark}

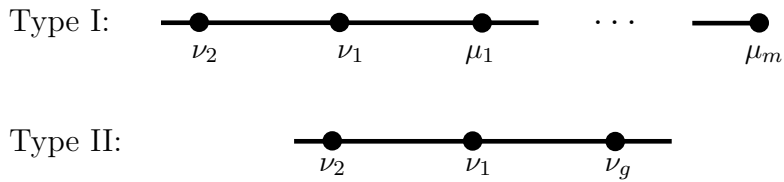
\begin{figure}[!ht]
  \centerline{
    \begin{overpic}[width=0.5\textwidth]{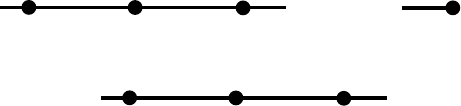}
      \put (-25,20) {\large Type~I:}
      \put (5,15.5) {$\nu_2$}
      \put (29,15.5) {$\nu_1$}
      \put (50,15.5) {$\mu_1$}
      \put (96,15.5) {$\mu_m$}
      \put (71,21) {\Large $\dots$}
      \put (-25,0) {\large Type~II:}
      \put (26,-3.3) {$\nu_2$}
      \put (50,-3.3) {$\nu_1$}
      \put (73,-3.3) {$\nu_g$}
    \end{overpic}
  }
\caption{Further structure of $\Gamma$ in the Type~I and Type~II cases.}
\label{fig:further_structure}
\end{figure}

The following lemma shows that under some circumstances the converse to Remark~\ref{rem:structure_consequences}\eqref{it:L_determines_Gamma} holds. It will be useful for recovering the Alexander polynomial from the structure of $\Gamma$.

\begin{lemma}\label{lem:large_q_rigidity}
Suppose that $q$ is a positive integer and $\Gamma$ is a plumbing graph with intersection form isomorphic to a $(N+\frac{1}{q})$-changemaker lattice $L$ for some integer $N \ge 0$. If $q$ is larger than number of vertices of $\Gamma$, then $L$ and, hence $N$, are uniquely determined by $\Gamma$.\footnote{The value of $N$ can also be determined by comparing the discriminant of both lattices.}
\end{lemma}
\begin{proof}
Since $\Gamma$ must have at least four vertices $q>2$. Thus there can be no vertices of type~(c), showing that the vertex set must be of Type~I or Type~II. Recall from Remark~\ref{rem:mu_i_useful_cases} that there are no vertices of the form $\mu_i$ when $p/q$ takes the form $\frac{p}{q}=n+\frac{1}{q}$. Thus by Lemma~\ref{lem:structure} we can assume that the vertices are the standard basis elements $\nu_1, \dots, \nu_t$. For $k>1$, we have
\[\norm{\nu_k}\leq k\leq t,\]
where the upper bound involving $k$ comes from observing that the largest possible norm of a non-tight standard basis element occurs when $\nu_k=-e_k+e_{k-1}+\dots + e_1$. However, using Lemma~\ref{lem:mui_cont_frac} we have that $\norm{\nu_1}=q+1$.
Therefore, the assumption that $q>t$ implies that $\nu_1$ is the unique vertex of norm $q+1$ in $\Gamma$. Now we can see inductively that the remaining vertices have unique embeddings as gapless standard basis elements. If we have a vertex $v$, whose image is not among $\nu_1, \dots, \nu_k$, but pairs with some $\nu_l$ for $l\leq k$, then $v$ must be embedded as $v=-e_g + e_{g-1}+ \dots + e_l$, where $g=l+\norm{v} - 1$, in order to ensure that $v\cdot \nu_l=-1$ and $v$ has the correct norm. Thus the choice of $\nu_1$ determines the rest of the embedding and hence the standard basis vectors of $L$. However, one can easily recover the structure of $L$ from its standard basis elements.
\end{proof}
The following example shows that the requirement that $q$ be sufficiently large is a necessary for the conclusion of Lemma~\ref{lem:large_q_rigidity} to hold.
\begin{eg}
The two $133/2$-changemaker lattices
\[
\langle f_1-f_0, f_0+e_1+e_2+e_3+2e_4+3e_5+5e_6+5e_7 \rangle^\bot
\]
and
\[
\langle f_1-f_0, f_0+e_1+e_2+2e_3+2e_4+2e_5+4e_6+6e_7 \rangle^\bot
\]
are both isomorphic to the same plumbing lattice. This can be seen by writing down the standard bases in each case. This example arises from the fact that $133/2$-surgery on $T_{5,13}$ and the $(2,33)$-cable of $T_{3,5}$ both yield the Seifert fibered space $S^2(2;\frac{13}{5},\frac{5}{3},\frac{3}{1})$.
\end{eg}

\subsection{The marked vertex}\label{sec:marked}
Now let $\Delta$ be a star-shaped or linear plumbing whose intersection form is isomorphic to an $(n-\frac{1}{2})$-changemaker lattice $L'$ by an isomorphism which carries the vertices of $\Delta$ to gapless standard basis elements of $L'$. We define the {\em marked vertex} of $\Delta$ to be the vertex of $\Delta$ which corresponds to $\nu_1=-e_1+f_0+f_1$. Note that this definition depends {\em a priori} on the lattice $L'$ and the choice of isomorphism. In practice, we will always have a fixed lattice $L'$ and a choice of isomorphism in mind, so it will be convenient to think of the marked vertex as being a property of $\Delta$. Although we will be primarily interested in the case where $\Delta$ is a star-shaped plumbing with $e=2$, we extend the definition to include the degenerate case that $\Delta$ is a linear plumbing as these will arise in the course of some ensuing proofs.
\begin{eg}\label{eg:marked}
Consider the $\frac{17}{2}$-changemaker lattice
\[
L'=\langle f_1 -f_0, f_0 + e_1 +e_2+e_3+e_4 + 2e_5 \rangle^\bot
\]
The standard basis elements for this lattice are $\nu_1 = -e_1 + f_0 + f_1$, $\nu_2=-e_2+e_1$, $\nu_3=-e_3+e_2$, $\nu_4=-e_4+e_3$ and $\nu_5=-e_5 + e_4 + e_3$. These are gapless and form the set of vertices for a plumbing $\Delta$ shown in Figure~\ref{fig:example_marked}.

\begin{figure}[!ht]
  \centerline{
    \begin{overpic}[width=0.4\textwidth]{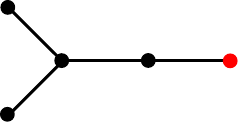}      
      \put (-20,25) {\LARGE $\Delta$}
      \put (-8,50) {$\nu_4$}
      \put (25,15.5) {$\nu_3$}
      \put (60,15.5) {$\nu_2$}
      \put (96,15.5) {$\nu_1$}
      \put (-8,0) {$\nu_5$}
      \put (8,50) {$2$}
      \put (25,30) {$2$}
      \put (60,30) {$2$}
      \put (96,30) {$3$}
      \put (8,0) {$3$}
    \end{overpic}
  }
\caption{The plumbing corresponding to Example~\ref{eg:marked} with the marked vertex indicated in red.}
\label{fig:example_marked}
\end{figure}
\end{eg}

For each changemaker lattice isomorphic to the intersection form of a plumbing graph $\Gamma$ with $e=2$, we will produce a plumbing graph $\Delta$ whose intersection form is isomorphic to a half-integer changemaker lattice with vertices mapping to gapless standard basis elements such that $\Gamma$ is obtained by modifying $\Delta$ near its marked vertex. We will then use this $\Delta$ to construct a knot in $S^3$ which surgers to give the Seifert fibered space corresponding to $\Gamma$.

First we show how to obtain an appropriate $\Delta$. In the Type~I and Type~II case this is an easy consequence of Lemma~\ref{lem:structure}. Recall that the stable coefficients of a changemaker lattice are defined in Definition~\ref{def:CMlattice}.
\begin{lemma}\label{lem:local_mod_type_I/II}
  Let $L$ be a $p/q$-changemaker lattice, where $\frac{p}{q}=n-\frac{r}{q}$ with $1\leq r<q$. Suppose that $L$ isomorphic to the intersection form of a plumbing $\Gamma$ with $e=2$ and the vertex set is of Type~I or Type~II. Then the $(n-\frac{1}{2})$-changemaker lattice $L'$ with the same stable coefficients as $L$ is isomorphic to the intersection form of a plumbing $\Delta$, where the vertex set is of Type~I or Type~II. Moreover $\Gamma$ is obtained by replacing the marked vertex of $\Delta$ by a chain of vertices of weights $\norm{\nu_1}, \norm{\mu_1},\dots, \norm{\mu_m}$. See Figure~\ref{fig:type_I/II_mod}.
\end{lemma}
\begin{figure}[!ht]
  \centerline{
    \begin{overpic}[width=0.95\textwidth]{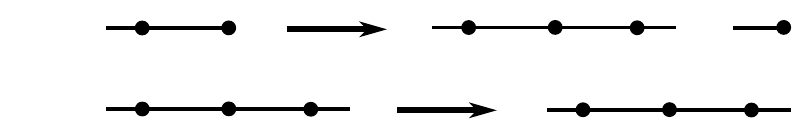}
      \put (0,12.5) {\large Type~I:}
      \put (17,10) {2}
      \put (27.5,10) {3}
      \put (57,10) {2}
      \put (66,10) {\small $\norm{\nu_1}$}
      \put (76,10) {\small $\norm{\mu_1}$}
      \put (85,13.3) {\Large $\dots$}
      \put (94,10) {\small $\norm{\mu_m}$}
      \put (0,2.5) {\large Type~II:}
      \put (17,0) {2}
      \put (27.5,0) {3}
      \put (37.5,0.2) {g}
      \put (71,0) {2}
      \put (80,0) {$q+1$}
      \put (92,0.2) {g}
    \end{overpic}
  }
\caption{Obtaining $\Gamma$ from $\Delta$ in the Type~I/II case. In both cases the marked vertices are the vertices of weight three in the plumbings on the left hand side.}
\label{fig:type_I/II_mod}
\end{figure}
\begin{proof}
  Let $\nu_1, \dots, \nu_t, \mu_1, \dots, \mu_m$ be the standard basis elements of $L$. By Lemma~\ref{lem:structure} we can assume that these are the vertices of $\Gamma$ and by the Type~I or Type~II assumption none of $\nu_2, \dots, \nu_t$ are tight. Thus the standard basis for $L'$ is
\[-e_1+f_0+f_1, \nu_2, \dots, \nu_t.\]
These standard basis elements pair exactly like the vertices of the plumbing graph $\Delta$ obtained from $\Gamma$ by deleting the vertices $\mu_1, \dots, \mu_m$ and changing the weight of $\nu_1$ to three. 
\end{proof}
The Type~III case is a little more subtle.
\begin{lemma}\label{lem:local_mod_type_III}
Let $L$ be a $p/q$-changemaker lattice, where $\frac{p}{q}=n-\frac{1}{q}$ and $q > 1$. Suppose that $L$ isomorphic to the intersection form of a plumbing $\Gamma$ with $e=2$ and the vertex set is of Type~III. Then the $(n+\frac{1}{2})$-changemaker lattice $L'$ with the same stable coefficients as $L$ is isomorphic to the intersection form of a plumbing $\Delta$, where the vertex set is of Type~II. Moreover $\Gamma$ is obtained by increasing the weight of the two vertices adjacent to the marked vertex of $\Delta$ by one and converting the marked vertex to a chain of $q-2$ vertices of weight 2. See Figure~\ref{fig:type_III_mod}.
\end{lemma}
\begin{figure}[!ht]
  \centerline{
    \begin{overpic}[width=0.8\textwidth]{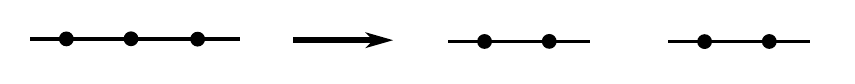}
      \put (-14,4.2) {\large Type~III:}
      \put (7,1) {2}
      \put (14.4,1) {3}
      \put (22,1.2) {$g$}
      \put (55,1) {3}
      \put (63,6) {$\overbrace{\makebox[80pt]{}}^{\text{$q-2$ vertices}}$}
      \put (63.5,1) {2}
      \put (71,4.5) {\Large $\dots$}
      \put (82,1) {2}
      \put (87,1) {$g+1$}
    \end{overpic}
  }
\caption{Obtaining $\Gamma$ from $\Delta$ in the Type~III case. The marked vertex is the vertex of weight three in the plumbing on the left hand side.}
\label{fig:type_III_mod}
\end{figure}
\begin{proof}
It will be convenient to write $L'$ as
\[
L'=\langle f_0+ e_0+\sigma_1 e_1+\dots+\sigma_t e_t, f_1-f_0  \rangle^\bot \subseteq \langle f_0,f_1,e_0, \dots, e_t\rangle=\Z^{t+3}.
\]
This differs only from the notation in Section~\ref{sec:CM_lattices} only by a shift in the indices on the $e_i$. We will show that $L'$ is isomorphic to the intersection form of the relevant plumbing. 

Let $\mu_1, \dots, \mu_m, v_1, \dots, v_t$ be the vertices of $\Gamma$, where we assume that $v_1=\nu_1$, and $v_2$ is the unique type~(c) vertex. By Lemma~\ref{lem:v_endtype},  we may assume that $v_2$ takes the form $v_2=-(\nu_k +\mu_1 + \dots + \mu_m)$, where $k>1$ and $\nu_k$ is tight. We modify these to obtain a collection of vectors $v_0', \dots, v_t'\in L'$ as follows. Take $v_0'=-e_0+f_0+f_1$, $v_1'=-e_1+e_0$, $v_2'= e_k-e_{k-1} - \dots - e_0$, and $v_k'=v_k$ for $k>2$. By construction we have that each of the $v_i'$ is in $L'$. 
\begin{claim}
The vectors $v_0', \dots, v_t'$ span $L'$.
\end{claim}
\begin{proof}[Proof of Claim]
Consider the standard basis $\nu_1, \dots, \nu_t$ for $L$. Since the standard basis elements for $L$ and the vertices of $\Gamma$ both form bases for $L$, there are integers $\alpha_{ik}, \beta_{jk}$ such that
\[
\nu_k=\sum_{i=1}^t \alpha_{ik} v_i + \sum_{j=1}^m \beta_{jk} \mu_j.
\]
Consider instead the vectors $\nu_1', \dots, \nu_t'$ in $L'$ defined by
\[\nu_k'=\sum_{i=1}^t \alpha_{ik} v_i'\]
By construction we have for all $j\geq 1$ that $v_j\cdot e_i=v_j'\cdot e_i$ for $i\geq 1$ and $v_j\cdot f_0 = v_j'\cdot e_0$.
Thus we see that $\nu_k'=\nu_k$ unless $\nu_k$ is tight, in which case $\nu_k' = -e_k + e_{k-1}+\dots + e_0$. In either case, we see that up to reindexing the $e_i$ to agree with the notation in Section~\ref{sec:CM_lattices}, the vectors $v_0',\nu_1', \dots, \nu_t'$ are precisely the standard basis vectors for $L'$. Since they are a linear combination of the $v_i'$, this proves that the $v_i'$ span $L'$.
\end{proof}
Let $\Delta$ be the plumbing graph obtained by replacing the linear chain in $\Gamma$ given by $v_1, \mu_1, \dots, \mu_m, v_2$ by the linear chain of vectors of norm $2, 3, \norm{v_2}-1$ respectively. By construction, the $v_i'$ almost pair as the vertices of $\Delta$: the only exception being that $v_2' \cdot v_0'=1$. However as $\Delta$ is a tree, we can choose signs $\varepsilon_i=\pm 1$ such that $\varepsilon_0=\varepsilon_1=1$, $\varepsilon_2=-1$ and $\varepsilon_0 v_0', \dots, \varepsilon_r v_t'$ pair as the vertices of $\Delta$. Thus as the $v_i'$ span $L'$ we see that the intersection form of $\Delta$ is isomorphic to $L'$. By construction the vertex set given by $\varepsilon_0 v_0', \dots, \varepsilon_t v_t'$ is of Type~II.
\end{proof}

Finally, we observe that changing the weight on a marked vertex to one results in a plumbing representing $S^3$. Figure~\ref{fig:example_blowdowns} illustrates how the plumbing from Example~\ref{eg:marked} blows down to the empty plumbing when the weight on the marked vertex is changed to one.
\begin{figure}[!ht]
  \centerline{
    \begin{overpic}[width=0.7\textwidth]{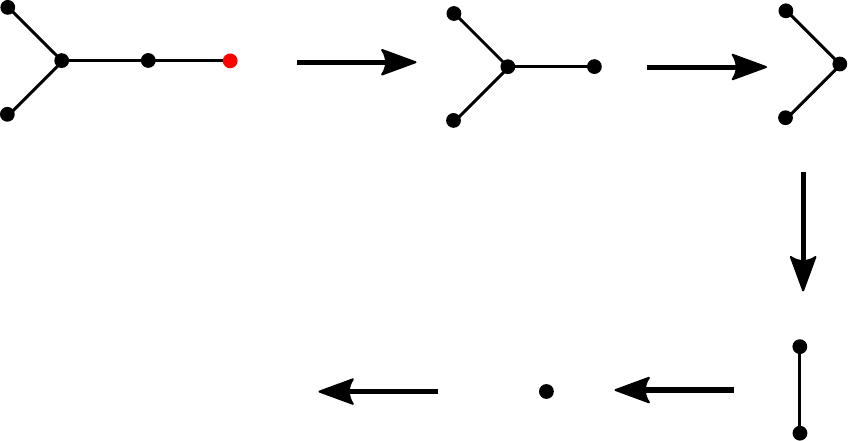}
      \put (-10,45) {\LARGE $\Delta'$}
      \put (-2,50) {$2$}
      \put (-2,38) {$3$}
      \put (7,47) {$2$}
      \put (17,47) {$2$}
      \put (27,47) {$1$}
      \put (50,50) {$2$}
      \put (50,38) {$3$}
      \put (59,47) {$2$}
      \put (69,47) {$1$}
      \put (90,50) {$2$}
      \put (90,38) {$3$}
      \put (99,47) {$1$}
      \put (90,10) {$1$}
      \put (90,0) {$2$}
      \put (60,5) {$1$}
       \put (30,5) {\LARGE $\emptyset$}
    \end{overpic}
  }
\caption{The sequence of blowdowns to the empty plumbing when the marked vertex of the plumbing $\Delta$ of Example~\ref{eg:marked} is changed to one.}
\label{fig:example_blowdowns}
\end{figure}
\begin{lemma}\label{lem:Gamma'_is_S3}
Let $\Delta$ be a star-shaped plumbing or a linear plumbing whose intersection form is isomorphic to a half-integer changemaker lattice $L$ by an isomorphism mapping vertices to gapless standard basis elements. Let $\Delta'$ be the plumbing obtained from $\Delta$ by changing the weight of the marked vertex to one. Then $\Delta'$ can be reduced to the empty plumbing by a sequence of blow-downs on weight 1 vertices.

In particular, the 4-manifold $X$ obtained by plumbing disk-bundles according to $\Delta'$ has boundary $\partial X\cong S^3$ and the corresponding surgery diagram for $S^3$ can be reduced to the empty diagram by performing a sequence of Rolfsen twists on 1-framed unknots.
\end{lemma}
\begin{proof} We will prove this inductively on the number of vertices in $\Delta$. Suppose that $\Delta$ is a tree whose intersection form is isomorphic to a half-integer changemaker lattice for which each vertex is a gapless standard basis element of $L$. When $L$ has rank one $\Delta$ consists of just a single vertex, $\nu_1$. The lemma is clearly true in this case. 

So now suppose that $L$ has rank $t>1$ and the vertices of $\Delta$ are gapless standard basis elements $\nu_1, \dots, \nu_t$. With the exception of $\nu_1$, these basis elements are not tight since they must have pairing $\nu_1 \cdot \nu_k \in \{0, -1\}$. Thus we must have $\sigma_2=1$ and $\nu_2=-e_2+e_1$. Note that any other vertex pairing with $\nu_1$ must take the form $\nu_g=-e_g+e_{g-1}+\dots + e_1$ for some $g>2$. If it exists then this $\nu_g$ is unique. For if we had $\nu_k=-e_k+e_{k-1}+\dots + e_1$ for some $k>g$, then
\[\nu_k\cdot \nu_g=g-1>0,\]
which is impossible for distinct vertices.

Thus if we obtain $\Delta'$ by changing the weight of the marked vertex $\nu_1$ to have weight one, we may perform a blow-down on this weight one vertex in $\Delta'$. This produces a new plumbing $\wt{\Delta}'$ with one fewer vertices. Since blowing down a weight one vertex decreases the weight of its neighbours by one, $\wt{\Delta}'$ contains a vertex of weight one. Let $\wt{\Delta}$ be the plumbing obtained by changing the weight of this vertex to three. These operations are illustrated in Figure~\ref{fig:inductive_step}.

The intersection form of $\wt{\Delta}$ embeds into the diagonal lattice generated by $e_2,\dots, e_t, f_0,f_1$, by taking vertices $\nu_2' ,\dots, \nu_t'$, where $\nu_2'=-e_2+f_0+f_1$, if there is $\nu_g=-e_g+e_{g-1}+\dots + e_1$, then $\nu_g'=-e_g+e_{g-1}+\dots + e_2$ and $\nu_k'=\nu_k$ for all other $k$.
However, these $\nu_2' ,\dots, \nu_t'$ are precisely the standard basis elements for some half-integer changemaker lattice
\[
L'=\langle w_0', f_1-f_0 \rangle^\bot \subseteq \langle f_0,f_1, e_2, \dots, e_t \rangle,
\]
of rank $t-1$ where $w_0'=f_0 +\sigma_2' e_2 + \dots + \sigma_t' e_t$ is defined by choosing  the $\sigma_i'$ inductively so that $\sigma_2'=1$ and $\sigma_k'$ is chosen to ensure that $\nu_k'\cdot w_0'=0$. Moreover these standard basis elements for $L'$ are gapless by construction. Thus we have an isomorphism from the intersection form of $\wt{\Delta}$ to a half-integer changemaker lattice which maps vertices to gapless standard basis elements. Moreover, the vertex corresponding to $\nu_2'$ is the marked vertex of $\wt{\Delta}$. Thus $\wt{\Delta}'$ is obtained by changing the marked vertex in $\wt{\Delta}$. Since $\wt{\Delta}$ has $t-1$ vertices, we can assume inductively that $\wt{\Delta}'$ can be blown-down to the empty diagram. Since $\wt{\Delta}'$ is obtained from $\Delta'$ by a blow-down it follows that $\Delta'$ can also be blown down to the empty plumbing, as required.

The statement about Rolfsen twists follows since a blow-down on the plumbing graph is achieved by a Rolfsen twist in the corresponding surgery diagram.
\begin{figure}[!ht]
  \centerline{
    \begin{overpic}[width=1\textwidth]{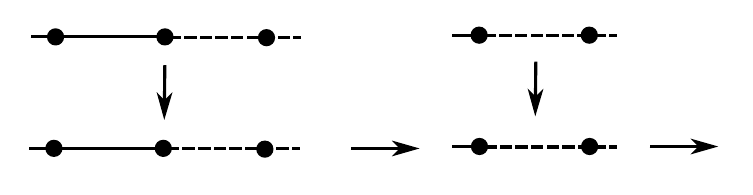}
    \put(-1,20){\Large $\Delta$:}
      \put (3,24) {\small $e_1-e_2$}
      \put (14,24) {\small $-e_1+f_0+f_1$}
      \put (30,24) {\small $-e_g+e_{g-1}+ \dots + e_1$}
      \put(-1,5){\Large $\Delta'$:}
      \put (7,2.2) {$2$}
      \put (21,2.2) {$1$}
	  \put (34.5,2.8) {$g$}      
      \put (63,2.2) {$1$}
      \put (76,2.8) {$g-1$}
      \put (55,24) {\small $-e_2+f_0+f_1$}
      \put (70,24) {\small $-e_g+e_{g-1}+ \dots + e_2$}
      \put (45,9) {\small blow-down}
      \put (85,9) {\small blow-down}
      \put (97,5) {\Large $S^3$}
    \end{overpic}
  }
 \caption{Showing inductively that $\Delta'$ blows down.}
\label{fig:inductive_step}
\end{figure} 
\end{proof}

\subsection{From lattices to surgeries}
Now we show how to pass from a changemaker lattices to knots with Seifert fibered space surgeries.
\begin{figure}[!ht]
  \centerline{
  \begin{overpic}[width=0.9\textwidth]{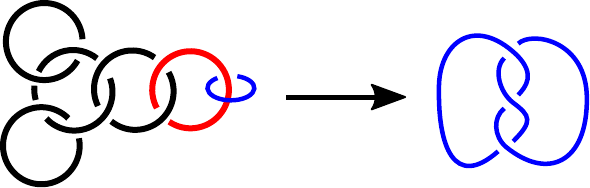}    
      \put (-2,25) {$3$}
      \put (16,7) {$2$}
      \put (-2,5) {$2$}
      \put (25,7) {$2$}
      \put (35,7) {$1$}
      \put (40,12) {$\frac{\alpha}{\beta}$}
      \put (50,18) {blowdowns}
       \put (75, 28) {$-(9-\frac{\alpha}{\beta})$}
       \put (42, 20) {$C$}
       \put (68,22) {\Large $K'$}
    \end{overpic}
 }
\caption{The construction of Lemma~\ref{lem:lattices_to_surgery} applied to the plumbing from Example~\ref{eg:marked}. After performing the necessary blowdowns the curve $C$ becomes the trefoil and the $\frac{\alpha}{\beta}$ surgery coefficient becomes $\frac{\alpha}{\beta}-9$.}
\label{fig:example_construction}
\end{figure}

\begin{lemma}\label{lem:lattices_to_surgery}
Let $\Gamma$ be a plumbing graph with $e=2$ whose intersection form is isomorphic to a $p/q$-changemaker lattice $L$, where $p/q \in \Q\setminus\Z$. If $Y$ is the corresponding Seifert fibered space, then there is a knot $K'$ which is either a torus knot or a cable of a torus knot such that $S_{-p/q}^3(K')\cong Y$ and the Alexander polynomial of $K'$ is determined by the stable coefficients of $L$.\footnote{That is to say that the torsion coefficients of $\Delta_{K'}(t)$ can be computed from $L$ by \eqref{eq:w0formula2}. As in Remark~\ref{rem:Alexander_recovery}, this allows us to calculate $\Delta_{K'}(t)$ from $L$ (cf. also Lemma~\ref{lem:CM_thm_plumbings}).}
\end{lemma}
\begin{proof} First consider the following construction.
Let $\Delta$ be a plumbing isomorphic to a $(n-\frac{1}{2})$-changemaker lattice $L'$ with the same stable coefficients as $L$ and with vertices of Type~I or Type~II. Note here that $n$ is the integer $n=\lceil p/q \rceil$. By Lemma~\ref{lem:structure}, we can assume that the vertices of $\Delta$ in $L'$ are gapless standard basis vectors and $\Delta$ has a marked vertex as defined at the start of Section~\ref{sec:marked}. Let $\Delta'$ be the plumbing obtained by changing the weight of the marked vertex in $\Delta$ to one and let $D$ be the surgery diagram corresponding to $\Delta'$. By Lemma~\ref{lem:Gamma'_is_S3}, $D$ is a surgery diagram for $S^3$.  Thus if we let $C$ be the meridian of the unique 1-framed unknot in $D$, then $C$ describes a knot $K'\subseteq S^3$. Note that even though $C$ is unknotted in the diagram $D$, the knot $K'$ will be non-trivial in general (see, for example, Figure~\ref{fig:example_construction}).

\begin{figure}[!ht]
  \centerline{
    \begin{overpic}[width=1\textwidth]{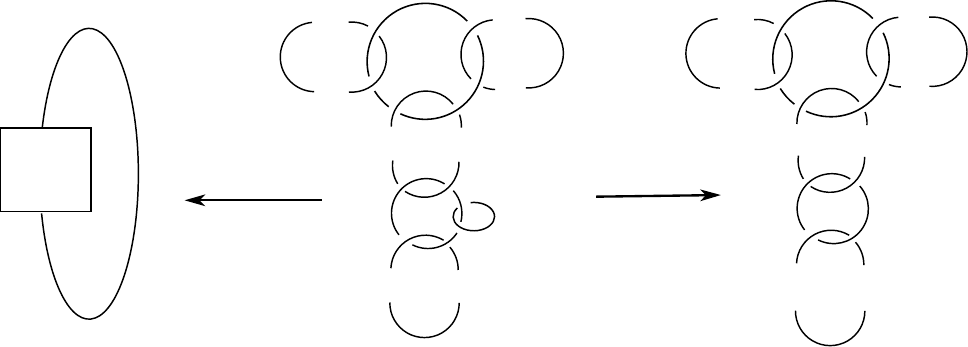}
      \put (3,17) {\Large $K'$}
      \put (14,29) {$-(N-\frac{\alpha}{\beta})$}
      \put (38,13) {$1$}
      \put (52,13) {$\frac{\alpha}{\beta}$}
      \put (91,13) {$1- \frac{\beta}{\alpha}$}
      \put (19,12) {Rolfsen twists}
      \put (63,12.5) {slam dunk}
      \put (44,5) {$\vdots$}
      \put (85.5,5) {$\vdots$}
      \put (44,20) {$\vdots$}
      \put (85.5,20) {$\vdots$}
      \put (33,30) {$\dots$}
      \put (52,30) {$\dots$}
      \put (75,30) {$\dots$}
      \put (94,30) {$\dots$}
    \end{overpic}
  }
\caption{The knot $K'$.}
\label{fig:knot_construction}
\end{figure}

Let $Y'$ be the 3-manifold obtained by performing $\frac{\alpha}{\beta}$-surgery on $C$ for some $\frac{\alpha}{\beta}\in \Q$. By Lemma~\ref{lem:Gamma'_is_S3}, we may perform a sequence of Rolfsen twists on 1-framed unknots to obtain a surgery description of $Y'$ involving only the component given by $C$ (i.e we obtain the surgery description for $Y$ in terms of the knot $K'$). Since each such Rolfsen twist decreases the framing on $C$ by a non-negative integer, we see that $Y'\cong S_{-(N-\frac{\alpha}{\beta})}^3(K')$ for some integer $N>0$ which is independent of $\frac{\alpha}{\beta}$.

Now consider the special case where $\frac{\alpha}{\beta}=-1/d$ for $d\geq 2$. In this case we may perform a slam dunk on the component $C$, to obtain a framing of $1+d$ on the component with which $C$ is linked. Observe that this is the surgery diagram corresponding to the plumbing graph $\Delta_d$ obtained by changing the weight of the marked vertex of $\Delta$ to have weight $d+1$. If $X_d$ is the plumbed 4-manifold corresponding to $\Delta_d$, then we have that
\[S_{-(N+\frac{1}{d})}^3(K')\cong \partial X_{d}.\]
It follows from Lemma~\ref{lem:CM_thm_plumbings} that the intersection form of $\Delta_d$ is isomorphic to a $(N+\frac{1}{d})$-changemaker lattice whose stable coefficients compute the Alexander polynomial $\Delta_{K'}(t)$. However, the intersection form of $\Delta_d$ is isomorphic to the $(n-1+\frac{1}{d})$-changemaker lattice with the same stable coefficients as $L'$. This isomorphism can be seen by observing that the standard basis elements of this $(n-1+\frac{1}{d})$-changemaker lattice form a set of vertices for the plumbing $\Delta_d$  (cf. Lemma~\ref{lem:local_mod_type_I/II}). Since $d$ can be taken to be arbitrarily large, it follows from Lemma~\ref{lem:large_q_rigidity} that $N=n-1$ and the Alexander polynomial of $K'$ is computed from the stable coefficients of $L'$. Moreover as all these surgeries are Seifert fibered spaces, Proposition~\ref{prop:q>9_case} implies that $K'$ is either a torus knot or a cable of a torus. With this construction in hand we prove the lemma.

\textbf{Case: Type~I or Type~II.} Suppose that $L$ is of Type~I or Type~II. Write $\frac{p}{q}=n-\frac{r}{q}$, where $1\leq r<q$. The standard basis elements $\nu_1,\mu_1, \dots, \mu_m$ of $L$ form a chain of vertices in $\Gamma$. Take $L'$ to be the $(n-\frac{1}{2})$-changemaker lattice with the same stable coefficients as $L$. By Lemma~\ref{lem:local_mod_type_I/II}, $L'$ is isomorphic to the intersection form of the plumbing $\Delta$ obtained by deleting $\mu_1,\dots, \mu_m$ and changing the weight on $\nu_1$ to be three. Let $K'$ be the knot constructed from $L'$ as in the first part of this proof. We have shown that the Alexander polynomial of $K'$ is determined by the stable coefficients of $L$ and that $K'$ is either a torus knot or a cable of a torus knot. It remains to check that $S_{-p/q}^3(K')\cong Y$.
 We obtain a surgery diagram for $S_{-p/q}^3(K')$ by taking the diagram $D$ and performing $(\frac{r}{q}-1)$-surgery on the meridian of $C$. Performing a slam-dunk allows us to absorb $C$ in to the 1-framed component and replace the framing on this component by
\[1+\frac{q}{q-r}= 1-\frac{1}{\frac{r}{q}-1}.\]
In the Type~II case, we have $\frac{r}{q}=\frac{q-1}{q}$. Thus after performing this slam-dunk we obtain a $(1+q)$-framed component, giving us the surgery diagram corresponding to the plumbing $\Gamma$ (see Figure~\ref{fig:type_II_slamdunks}). This shows that $S_{-p/q}^3(K')$ is the required Seifert fibered space in the Type~II case.

In the Type~I case, we perform a sequence of reverse slam-dunks to obtain an integer surgery diagram. Using Lemma~\ref{lem:mui_cont_frac} and $\norm{\nu_1}=1+\norm{\mu_0}$ we see that 
\[
1+\frac{q}{q-r}=[\norm{\nu_1},\norm{\mu_1},\dots, \norm{\mu_m}]^-.
\]
Thus if we perform a sequence of reverse slam-dunks to convert this to a surgery diagram with integer coefficients then this gives a chain of unknots with surgery coefficients $\norm{\nu_1},\norm{\mu_1},\dots, \norm{\mu_m}$. This is illustrated in Figure~\ref{fig:type_I_slamdunks}. However this surgery diagram is precisely the surgery diagram for $Y$ corresponding to $\Gamma$, so we have shown that $S_{-p/q}^3(K')$ is the required Seifert fibered space.

\begin{figure}[!ht]
  \centerline{
    \begin{overpic}[width=1\textwidth]{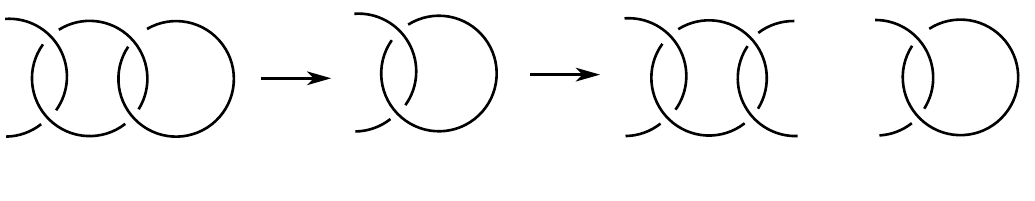}
      \put (67,3.7) {\small $\norm{\nu_1}$}
      \put (75,3.7) {\small $\norm{\mu_1}$}
      \put (17,3.5) {$\frac{r}{q}-1$}
      \put (42,4.7) {$1+\frac{q}{q-r}$}
      \put (9,3.5) {$1$}
      \put (95,4.5) {\small $\norm{\mu_m}$}
       \put (79,11.5) {\Large $\dots$}
    \end{overpic}
  }
\caption{Surgery calculus in the Type~I case.}
\label{fig:type_I_slamdunks}
\end{figure}
\begin{figure}[!ht]
\vspace{10pt}
  \centerline{
    \begin{overpic}[width=0.5\textwidth]{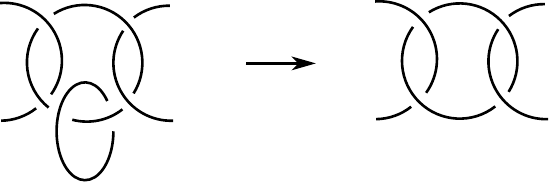}
      \put (69,4) {\small $\norm{\nu_1}=1+q$}
      \put (17,-2) {$-\frac{1}{q}$}
      \put (14,34) {$1$}
    \end{overpic}
  }
\caption{Surgery calculus in the Type~II case, where $r=q-1$.}
\label{fig:type_II_slamdunks}
\end{figure}

\textbf{Case: Type~III.} When the vertices of $\Gamma$ are of Type~III and $\frac{p}{q}=n-\frac{1}{q}$, take $L'$ be the $(n+\frac{1}{2})$-changemaker lattice with the same stable coefficients as $L$. By Lemma~\ref{lem:local_mod_type_III} this is isomorphic to the intersection form of a plumbing $\Delta$ with Type~II vertices.

Let $K'$ be the knot constructed from $L'$ as in the first part of the proof. Such a knot is either a torus knot or a cable of a torus knot and has the required Alexander polynomial. Thus it remains only to check that it has the desired surgery. We obtain a surgery diagram for $S_{-p/q}^3(K')$ by performing $\frac{1}{q}$-surgery on the curve $C$. By performing a slam dunk, this can be absorbed to a give a $(1-q)$-framed unknot. This results in a chain of unknotted components with framings $2$, $1-q$ and $g$, respectively for some $g$. By performing a sequence of $q-2$ blow-ups introducing 1-framed components, we can increase the $1-q$ framing to $-1$. When can we blow this $-1$- components down to obtain a chain of unknots with every framing at least two. The result of these operations is to replace the chain with weights $2,1-q,g$, by a chain with weights
\[3, \underbrace{2,\dots, 2}_{q-2}, g+1.\]
This is shown in Figure~\ref{fig:type_III_surgery_calc}. However this diagram is precisely the surgery diagram for $Y$ corresponding to $\Gamma$. Thus we have shown that $S_{-p/q}^3(K')$ is the required Seifert fibered space.   

\begin{figure}[!ht]
  \centerline{
    \begin{overpic}[width=1\textwidth]{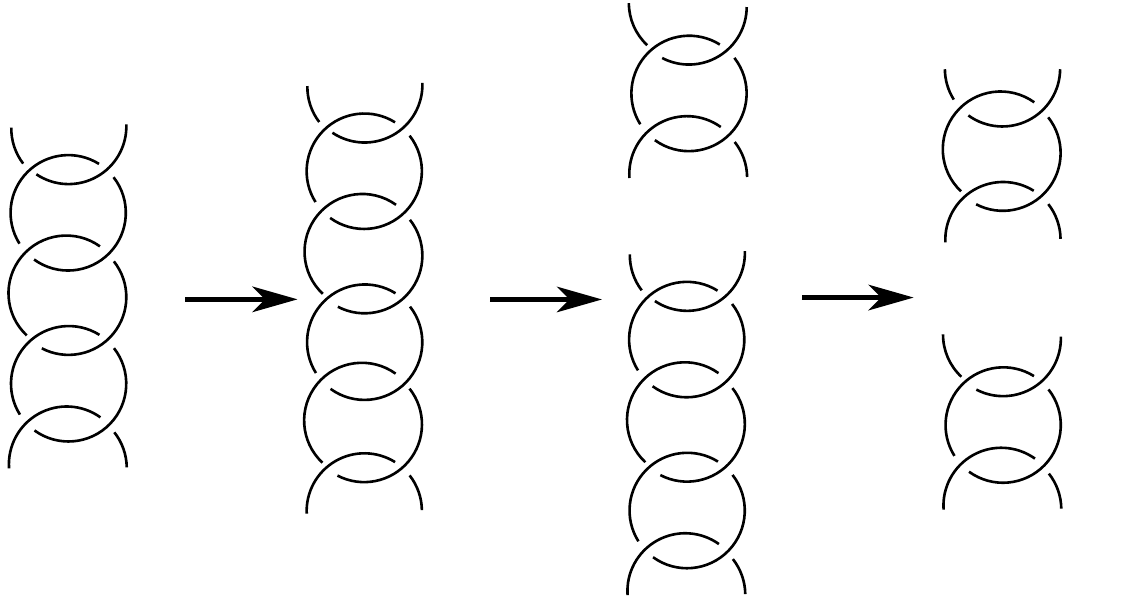}
      \put (12,35) {2}
      \put (11.5,29) {$1-q$}
      \put (12,20) {$g$}
      \put (38,38) {3}
      \put (38,30) {1}
      \put (38,22) {$2-q$}
      \put (38,16) {$g$}      
      \put (66,44) {3}
      \put (66,38) {2}
      \put (59.5,33) {\Large$\vdots$}
      \put (66,29) {2}
      \put (66,23) {1}
      \put (65,16.5) { $-1$}
      \put (65,9.5) { $g$}
      \put (94,37) {3}
      \put (94,32) {2}
      \put (87.5,26) {\Large$\vdots$}
      \put (94,22) {2}
      \put (94,16) {$g+1$}
      \put (95, 27) {$\left. \rule{0cm}{33pt} \right\rbrace$}
      \put (97.2, 27.9) {$q-2$}
      \put (97.2,25.3) {unknots}
    \end{overpic}
  }
\caption{Surgery calculus in the Type~III case.}
\label{fig:type_III_surgery_calc}
\end{figure}
\end{proof}

\begin{remark} Some observations on the preceding lemma are in order.
\begin{enumerate}[(i)]
\item Although we used Proposition~\ref{prop:q>9_case} to deduce that the knot $K'$ is a torus knot or a cable of a torus knot, one can also deduce this fact directly by studying how the curve $C$ sits inside the surgery diagram for $S^3$.
\item One can check that the knot $K'$ constructed in the previous lemma is a torus knot in the Type~I case and a cable of a torus knot in the Type~II and Type~III cases.
\end{enumerate}
\end{remark}
\section{Analysis for the $e\geq 3$ case}\label{sec:alt_case}
In this section, we develope the methods to prove Theorem~\ref{thm:nonint} for $e\geq 3$. In this case the surgered Seifert fibered space is the branched double cover of an alternating Montesinos link. This allows us to apply results of \cite{mccoy2014noninteger} and \cite{mccoy2016thesis} which characterize when the double branched cover of an alternating link can arise by non-integer surgery. Before we state these results we will set out some conventions.

A {\em tangle} $T=(B^3,A)$ will always be a properly embedded 1-manifold $A$ in $B^3$ where $\partial B^3 \cap A$ consists of four points. Thus the branched double cover of a tangle $T$ will always be a 3-manifold with torus boundary. When considering isotopies between tangles, we will allow isotopies that move $\partial B^3$. In particular, we will allow isotopies that exchange boundary points of $A$. If two tangles $T$ and $T'$ are isotopic, then their double branched covers are homeomorphic. For the purposes of this paper, one may take a {\em rational tangle} to simply mean a tangle whose double branched cover is a solid torus. The notion of slope for rational tangles will not be used.

A {\em Conway sphere} for a knot $K$ is an embedded sphere in $S^3$ intersecting the knot transversely in four points. A Conway sphere is said to be {\em visible} in a diagram if it intersects the plane of the diagram in a connected simple closed curve and intersects the diagram transversely in four points. Note that a Conway sphere always separates a diagram into two tangles.  

The following is an amalgamation of Theorem~7.1 and Theorem~7.12 of \cite{mccoy2016thesis}.

\begin{thm}\label{thm:alt_surgery}
Let $L$ be an alternating knot or link such that $S_{p/q}^3(K)\cong \Sigma(L)$ for some knot $K\subseteq S^3$ and $p/q \in \Q\setminus \Z$. Then $L$ has a reduced alternating diagram $D$ with a visible Conway sphere $C$ which separates $D$ into two tangles such that:
\begin{enumerate}[(i)]
\item one tangle is a rational tangle containing at least one crossing which can be replaced with a single crossing to obtain an almost-alternating diagram of the unknot and
\item the other tangle is such that its branched double cover is homeomorphic to the complement of a knot $K'\subseteq S^3$ with $\Delta_K(t)= \Delta_{K'}(t)$ and $S_{p/q}^3(K')\cong S_{p/q}^3(K)\cong \Sigma(L)$.
\end{enumerate}
\end{thm}
Recall that an almost-alternating diagram is one that can be transformed into an alternating diagram by changing a single crossing.
Although Theorem~\ref{thm:alt_surgery} only guarantees the existence of a single diagram for $L$ with a nice Conway sphere, we can easily obtain a similar condition on any alternating diagram of $L$. This uses the fact that any two reduced alternating diagrams of the same alternating link are related by flypes and planar isotopy \cite{Menasco93classification}. See Figure~\ref{fig:flype} for an example of a flype.

\begin{prop}\label{prop:alt_surgery_refined}
Let $L$ be an alternating knot or link such that $S_{p/q}^3(K)\cong \Sigma(L)$ for some knot $K\subseteq S^3$ and $p/q \in \Q\setminus \Z$. Then for any reduced alternating diagram $D$ of $L$ there a visible Conway sphere $C$ separating $D$ into two tangles such that:
\begin{enumerate}[(i)]
\item one tangle is a single crossing
\item the other tangle is such that its double branched cover is homeomorphic to the complement of a knot $K'\subseteq S^3$ with $\Delta_K(t)= \Delta_{K'}(t)$ and $S_{p/q}^3(K')\cong S_{p/q}^3(K)\cong \Sigma(L)$.
\end{enumerate}
\end{prop}
\begin{proof}
First we will show that there is some reduced alternating diagram for $L$ with the required property. To do this take the diagram $D$ of $L$ along with the Conway sphere $C$ guaranteed by Theorem~\ref{thm:alt_surgery}. The rational tangle side of $C$ contains at least one crossing. We will show that if $C$ contains more than one crossing, then it can be `shrunk' until it contains a single crossing. It follows from the results
of \cite[Section~4]{Kauffman2004rational} that in any alternating diagram of a rational tangle at least one pair of
arcs emerging from the boundary sphere must meet in a crossing. 

\begin{figure}[!ht]
  \centerline{
    \begin{overpic}[width=0.7\textwidth]{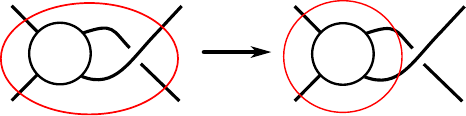}
      \put (-5,11) {\Large $C$}
      \put (85,22) {\Large $C'$}
    \end{overpic}
  }
 \caption{Shrinking $C$ to obtain $C'$.}
\label{fig:shrinking_C}
\end{figure}

Thus we can assume that $C$ appears as in Figure~\ref{fig:shrinking_C}. Take $C'$ to be the Conway sphere obtained by shrinking $C$ to omit this crossing. Notice that the tangles on the outside of $C$ and $C'$ are isotopic by an isotopy swapping the two right-most end points to eliminate a crossing. Thus we see that the branched cover of the exterior of $C'$ is still the knot complement $S^3\setminus \nu K'$. Continuing this way we can reduce $C$ until it contains a single crossing, thus giving a Conway sphere in $D$ with the required properties.  

\begin{figure}[!ht]
  \centerline{
    \begin{overpic}[width=0.5\textwidth]{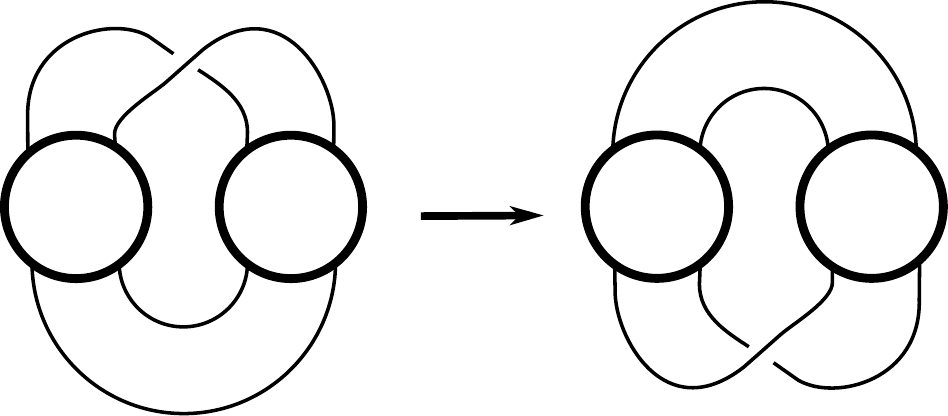}
      \put (6,20) {\Large $F$}
      \put (29,20) {\Large $B$}
      \put (67,20) {\Large \reflectbox{$F$}}
      \put (90,20) {\Large $B$}
    \end{overpic}
  }
 \caption{A flype.}
\label{fig:flype}
\end{figure}

\begin{figure}[!ht]
  \centerline{
    \begin{overpic}[width=0.5\textwidth]{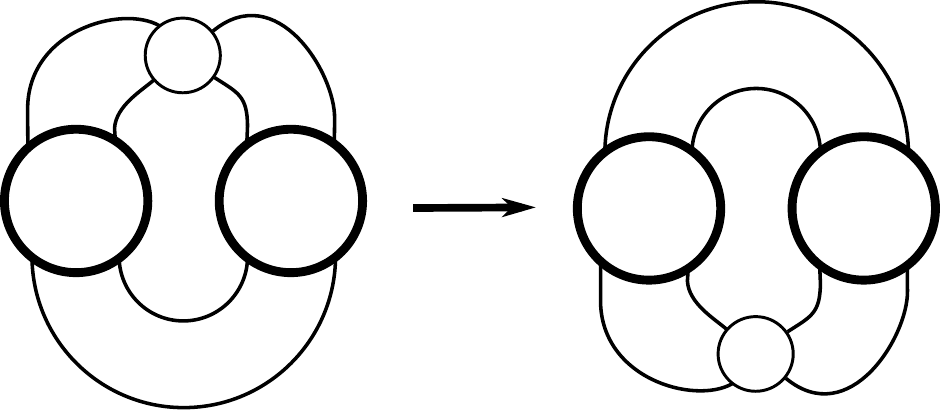}
      \put (6,20) {\Large $F$}
      \put (29,20) {\Large $B$}
      \put (67,20) {\Large \reflectbox{$F$}}
      \put (90,20) {\Large $B$}
    \end{overpic}
  }
 \caption{The choice of $C$ ad $C'$ after flyping.}
\label{fig:flyped_crossing}
\end{figure}

Thus suppose that we have a diagram $D$ with a Conway sphere $C$ with the desired properties. Now let $D'$ be any other reduced alternating diagram for $L$. This can be obtained from $D$ by a sequence of planar isotopies and flypes. It is clear that planar isotopies preserve the required property so we only need to check that the existence of $C$ is preserved under flypes. Consider a flype as depicted in Figure~\ref{fig:flype}. When $C$ is contained in one of the tangles marked $F$ or $B$, then it is clear that the image of $C$ under the flype will again be a Conway sphere with the required properties. Thus we need only consider the case that $C$ encloses the crossing destroyed by the flype. In this case we take $C'$ to be a Conway sphere in $D'$ containing only the crossing created by the flype. See Figure~\ref{fig:flyped_crossing}. However consider the tangles on the outside of $C$ and $C'$. It is not hard to see that these tangles are related by a sequence of isotopies and mutations. Since isotopies and mutations do preserve the homeomorphism type of the double branched cover $C'$ has the required properties. 
\end{proof}



Combining Proposition~\ref{prop:alt_surgery_refined} with Proposition~\ref{prop:q>9_case} allows us to prove Theorem~\ref{thm:nonint} for $e\geq 3$.

\begin{lemma}\label{lem:e>=3_case}
Let $Y=S^2(e; \frac{p_1}{q_1}, \frac{p_2}{q_2},\frac{p_3}{q_3})$ be a Seifert fibered space with $e\geq 3$ such that $S_{p/q}^3(K)\cong Y$ for some $K\subseteq S^3$ and $p/q \in \Q\setminus\Z$. Then there is a knot $K'\subseteq S^3$ which is either a torus knot or a cable of a torus knot with $S_{p/q}^3(K') \cong Y$ and $\Delta_K(t)=\Delta_{K'}(t)$.
\end{lemma}
\begin{proof}
It follows from the classification of Montesinos knots in terms of their double branched covers (see, for example, \cite[Chapter~12]{BurdeZieschang}) that such a $Y$ is the double branched cover of an alternating Montesinos link $L$ with 3 arms. Such a link has a diagram of the form $D$ shown in Figure~\ref{fig:Monty_diag}, where the rectangular boxes are twist regions each containing some number of crossings.

By Proposition~\ref{prop:alt_surgery_refined}, there is a Conway sphere $C$ containing on one side a single crossing $c$ and on the other a tangle such that the double branched cover of its exterior is homeomorphic to the complement of a knot $K'$ in $S^3$ such that $S_{p/q}^3(K')\cong \Sigma(L) \cong Y$ and with $\Delta_K(t) = \Delta_{K'}(t)$. Thus we need only to check that $K'$ is a torus knot or a cable of a torus knot.

The crossing $c$ contains lies in some twist region $R$ of $D$. For any $n>1$, let $D_n$ be the alternating diagram obtained by replacing $c$ with a twist region containing $n$ crossings, where we insert the crossings so that the twist region $R$ is extended in length. Since, we started with a diagram of the form shown in Figure~\ref{fig:Monty_diag} and extended the length of a twist region, we see that $D_n$ is still in the form given in Figure~\ref{fig:Monty_diag}. Thus $D_n$ is a Montesinos knot or link and its double branched cover $\Sigma(D_n)$ is a Seifert fibered space \cite[Chapter~12]{BurdeZieschang}. Since we obtained $D_n$ by replacing the crossing $c$ with a rational tangle, the Montesinos trick shows that there is a rational number $p_n/q_n \in \Q$ such that $S_{p_n/q_n}^3(K')\cong \Sigma(D_n)$ \cite{montesinos1975surgery}. Since the crossing numbers of the $D_n$ is monotonically increasing, we see that $D_n$ are diagrams for distinct knots or links. As there are only finitely many non-split alternating knots or links with a given determinant, we see that $\det D_n$ and hence $|p_n|$ tends to infinity. By \cite[Theorem~1.1]{mccoy2014bounds} any such $p_n/q_n$ satisfies $|p_n/q_n| \leq 4g(K')+3$. Thus for $n$ sufficiently large we have $q_n \geq 9$. Thus Proposition~\ref{prop:q>9_case} applies to show that $K'$ is either a torus knot or a cable of a torus knot, as required. 
\end{proof}

\begin{figure}[!ht]
  \centerline{
    \begin{overpic}[width=0.7\textwidth]{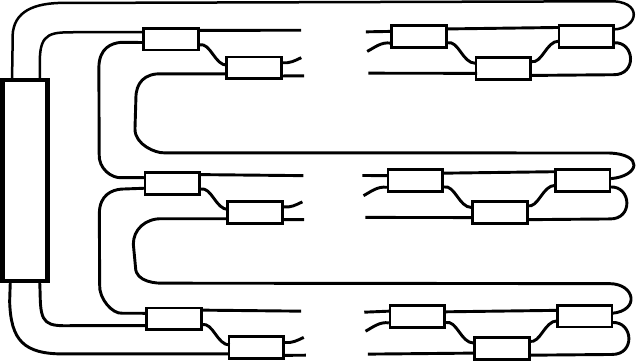}
      \put (-13,30) {$e-3$}
      \put (-15,27) {crossings}
      \put (49,4) {\Large $\dots$}
      \put (49,25) {\Large $\dots$}
      \put (49,48) {\Large $\dots$}
    \end{overpic}
  }
 \caption{A diagram for a three armed Montesinos link. The boxes represent twist regions.}
\label{fig:Monty_diag}
\end{figure}

\section{Proofs of Theorem~\ref{thm:nonint} and Proposition~\ref{prop:e_calc}}\label{sec:proof_of_main}
Thus we are ready to prove Theorem~\ref{thm:nonint} and conclude the paper.
\thmnonint*
\begin{proof}[Proof of Theorem~\ref{thm:nonint}]
We have two cases to consider either $e(Y)=-2$ or $|e(Y)|\geq 3$. First suppose that $e(Y)=-2$. Since 
\[
e(-Y)=2 \quad\text{and}\quad -Y\cong S^3_{-p/q}(\overline{K}),
\]
where $\overline{K}$ is knot obtained by reflecting $K$, Lemma~\ref{lem:CM_thm_plumbings} shows that the intersection form of the canonical plumbing bounding $-Y$ is isomorphic to a $p/q$-changemaker lattice whose stable coefficients compute the Alexander polynomial of $\overline{K}$. Lemma~\ref{lem:lattices_to_surgery} then shows that the existence of this changemaker lattice allows us to construct a knot $\overline{K'}$ which is a torus knot or a cable of a torus knot with $\Delta_{\overline{K'}}(t)=\Delta_{\overline{K}}(t)$ and $S_{-p/q}^3(\overline{K})\cong S_{-p/q}^3(\overline{K'})$. Reflecting yields a knot $K'$ with the desired properties.  After reflecting suitably, the case that $|e(Y)|\geq 3$ is precisely given by Lemma~\ref{lem:e>=3_case}.
\end{proof}

We now turn to the proof of Proposition~\ref{prop:e_calc}. First note that the Seifert invariants of surgeries on a torus knot can easily be calculated directly (see, for example, \cite{Moser71elementary} or  \cite[Lemma~4.4]{Owens12negdef}).
\begin{prop}\label{prop:torus_surg}
For coprime $r,s>1$, let $Y\cong S^3_{p/q}(T_{r,s})$. Then
\begin{enumerate}[(i)]
\item $Y$ is reducible if $p/q=rs$;
\item $Y$ is a lens space if $p/q = rs \pm \frac{1}{q}$;
\item otherwise $Y$ is the small Seifert fibered space with three exceptional fibers
 \[
Y\cong S^2\left(1;\frac{r}{s'},\frac{s}{r'},\frac{p}{q}-rs \right)
\]
where $s'$ and $r'$ are the integers satisfying $1\leq s'<r$, $1\leq r' <s$ and $\frac{s'}{r}+\frac{r'}{s}=1+\frac{1}{rs}$ .
\end{enumerate}
\end{prop}
The corresponding result for negative torus knots can be obtained by changing orientations, since $S_{p/q}^3(T_{r,s})\cong -S_{-p/q}^3(T_{-r,s})$. Next we calculate $e(Y)$ for these surgeries.

\begin{lemma}\label{lem:e(Y)_calc}
Let $Y\cong S_{p/q}^3(T_{r,s})$ be a small Seifert fibered space, where $r,s>1$. Then $e(Y)$ satisfies
\begin{enumerate}[(i)]
\item $e(Y)=-1$ if $p/q<0$;
\item $e(Y)=2$ if $0<p/q<rs-1$;
\item $e(Y)\geq 3$ if $rs-1<p/q<rs$ and
\item $e(Y)\leq -2$ if $p/q>rs$.
\end{enumerate}
\end{lemma}
\begin{proof}
By Proposition~\ref{prop:torus_surg}, we have that
\begin{equation}\label{eq:pos_form}
Y\cong S^2\left(1;\frac{r}{s'},\frac{s}{r'},\frac{p}{q}-rs\right).
\end{equation}
This shows that
\[
\varepsilon(S_{p/q}^3(T_{r,s}))= \frac{1}{rs}\left(\frac{p/q}{rs- p/q}\right),
\]
which implies that $\varepsilon(Y)>0$ if $0<p/q<rs$ and $\varepsilon(Y)<0$ if $p/q<0$ or $p/q>rs$.

First suppose that $0<p/q<rs$, so that $\varepsilon(Y)>0$. We apply Rolfsen twists to \eqref{eq:pos_form} to show that $Y$ takes the form 
\[
Y\cong S^2\left(1+n;\frac{r}{s'},\frac{s}{r'},\frac{p-rsq}{q+n(p-rsq)}\right)
\]
where $n$ is such that $\frac{p-rsq}{q+n(p-rsq)}>1$. One can check that the necessary value of $n$ is $n=\lceil \frac{q}{rsq-p} \rceil$. Thus we have that $e(Y)=2$ if $0<p/q<rs-1$ and $e(Y)\geq 3$ if $rs-1<p/q<rs$, as required.

By Rolfsen twisting twice we see that $Y$ can be written in the form
\begin{equation}\label{eq:neg_form}
Y\cong S^3_{p/q}(T_{r,s})\cong S^2\left(-1;-\frac{r}{r-s'},-\frac{s}{s-r'},\frac{p}{q}-rs\right)
\end{equation}
If $p/q<0$, then $\varepsilon(Y)<0$ and we have that $\frac{p}{q}-rs<-1$. Thus the description in \eqref{eq:neg_form} shows that $e(Y)=-1$ in this case. If $p/q>rs$, then $\varepsilon(Y)<0$, but $\frac{p}{q}-rs>0$. Thus by Rolfsen twisting we see that $Y$ takes the form
\[
Y\cong S^2\left(-1-n;-\frac{r}{r-s'},-\frac{s}{s-r'},\frac{p-qrs}{q-n(p-qrs)}\right)
\]
where $n\geq 1$ is chosen to ensure that $\frac{p-qrs}{q-n(p-qrs)}<-1$. This shows that $e(Y)\leq -2$ in this case.
\end{proof}

This allows us to determine the surgeries arising in the conclusion of Theorem~\ref{thm:nonint}.

\ecalc*
\begin{proof}
If $K$ is a torus knot, then this is a consequence of Lemma~\ref{lem:e(Y)_calc} and Proposition~\ref{prop:torus_surg}. The result is deduced for cables of torus knots by using the fact that Seifert fibered surgeries on $C_{a,b}\circ T_{r,s}$ take the form
\[
S^3_{ab \pm \frac{1}{q}}(C_{a,b}\circ T_{r,s})\cong S^3_{\frac{qab\pm 1}{qa^2}}(T_{r,s}),
\]
where $a$ is the winding number of the pattern torus knot (cf. \cite[Lemma 3.3]{Gordon83Satellite}). The condition that $b/a> rs-1$ is a consequence of the fact that $\frac{qab\pm 1}{qa^2}> rs-1$ if and only if $b/a > rs-1$.
\end{proof}
\phantomsection
\bibliography{qabib}
\bibliographystyle{alpha}

\end{document}